\RequirePackage{fix-cm}
\documentclass[smallextended]{svjour3}       
\smartqed  

\usepackage{graphicx}
\usepackage{mathptmx}      
\usepackage{amssymb,amsmath,amsfonts,enumerate,latexsym,mathrsfs}
\usepackage{latexsym}
\makeatletter
\def\@thmcontersep{.}
\makeatother

\usepackage{color}

\begin{document}

\title{Theory of generating spaces of convex sets and their applications to solvability of convex programs in Banach spaces}

\titlerunning{Theory of generating spaces and applications}        

\author{Lixin Cheng        \and
        Weihao Mao 
}


\institute{Department of Mathematics, Xiamen University \at
              Siming, Xiamen, 361005, Fujian, China \\
              \email{lxcheng@xmu.edu.cn}           
           \and
           Department of Mathematics, Xiamen University\at
             Siming, Xiamen, 361005, Fujian, China \\
             \email{mwhaea123456@163.com}
       }

\date{Received: date / Accepted: date}

\maketitle

\begin{abstract}
When optimization theorists consider optimization problems in infinite dimensional spaces, they need to deal with closed convex subsets(usually cones) which mostly have empty interior. These subsets often prevent optimization theorists from applying powerful techniques to study these optimization problems. In this paper, by nonsupport point, we present generating spaces which are relative to a Banach space and a nonsupport point of its convex closed subset. Then for optimization problems in infinite dimensional spaces, in some general cases, we replace original spaces by generating spaces while containing solutions. Thus this method enable us to apply powerful classical techniques to optimization problems in very general class of infinite dimensional spaces.
Based on functional analysis, from classical Banach spaces to separable Banach spaces, from Banach lattice to latticization, we give characterizations of generating spaces and conclude that they are actually linearly isometric to $L_\infty$($\ell _\infty$) or their closed subspaces. Thus continuous linear functional involved in these techniques could be chosen from $L_\infty^*$($\ell_\infty^*$). After that, applications in Penalty principle, Lagrange duality and scalarization function are further studied by this method.
\end{abstract}

\section{Introduction}\label{sec1}
In infinite dimensional Banach spaces, we always deal with cones. Solidness assumption of these target cones is essential in many cases, including duality theory, variational analysis and so on. A convex subset of a Banach space $X$ is said to be solid if it has nonempty interior. A challenge in these programs is that these cones are not solid in most cases. For example, if the space in question is an infinite dimensional  $L_p(\Omega,\mu)$, or, $\ell_p$  ($1\leq p< \infty$), then its natural positive cone  $L_p^+(\Omega,\mu)$,  or, $\ell_p^+$, has  empty interior.
Indeed, every Banach space admitting a reproducing cone with nonempty interior is ``close" to a $C(K)$-space (see, Theorem \ref{8.5}). This means that if $X$ is a Banach space admitting a reproducing cone with nonempty interior, then
there is a Banach space $C(K)$ for some compact Hausdorff space $K$ such that $X$ is isomorphic to a subspace $E$ of $C(K)$ satisfying that the ``$C(K)$-lattice hull" of $E$ is dense in $C(K)$. In other words, the class of $C(K)$-spaces is almost the only class of Banach spaces  admitting a reproducing cone with nonempty interior. Therefore, it has become a significantly  important work to find an appropriate substitute for the interior of a reproducing cone in an infinite dimensional Banach space.\\

The Bishop-Phelps theorem states that for a nonempty closed convex set $C$ in a Banach space $X$, support functionals of $C$ are dense in the cone $C^*\subset X^*$ consisting of all functionals which are bounded above on $C$; and support points of $C$ are always dense in the boundary of $C$. (See, for instance, \cite{ph}.) We denote by $C_S$, the set of all support points of $C$, and by $C_N=C\setminus C_S$, the set of all non-support points of $C$. Then by the separation theorem of convex sets it is easy to see that $C_N={\rm int}C$ (the interior of $C$) if the latter is nonempty. In 1974, R.R. Phelps \cite{72Phelps} further studied topological properties of  $C_S$ of $C$, and obtained that $C_S$ is always a $F_\sigma$ set of $C$, and its nonsupport point set $C_N$   is a dense $G_\delta$-subset of $C$ whenever $C_N\neq\emptyset$. Taking $C_N$ as a starting point, many mathematicians used this notion to study  differentiability of real-valued convex, or locally Lipschiz functions defined on  closed convex sets $C$ with int$C=\emptyset$ but $C_N\neq\emptyset$ and call such a set "small set" \cite{88Verona,re,94Wu}, to show some "embedding" theorems \cite{CC,09Cheng,CZh},  to study ``lattice" properties of induced by a cone $C$ with  $C_N\neq\emptyset$ \cite{99Schaefer},  and to characterize a cone in a Banach space with a base \cite{16Cheng,75Homles}. Coincidentally, \cite{03Borwein,91Limber,93Limber} applied nonsupport point subset of a convex set (but they call it ``quasi interior") to solvability discussion of convex program, vector optimization and to  convex duality theory. For more information in this direction, we refer the reader to \cite{15Zalinescu,05Cammaroto,08Bot,12Bot,14Grad,21Cuong} and references therein.\\

From the facts mentioned above, for a closed convex set $C$ it seems that  such a substitute (``$C_N$" for the interior int$C$)  has  already exist. Therefore, the question that how to solve the following new problems has moved toward into the central stage.

\begin{problem}\label{1.1}
Assume that $C$ is a closed convex subset  of a Banach space $(X,\|\cdot\|)$ with empty interior but $C_N\neq\emptyset$.

 I. How to produce a new Banach space $X_C$ so that

 i) $X_C$ is algebraically contained in $X$;

 ii)  $C_0\equiv X_C\bigcap C$  is dense in $C$ with respect to the norm  of $X$; and

 iii)  ${\rm int}_{X_C}C_0$ (the interior of $C_0$) is nonempty?

 II. How to represent the new space $X_C$?
\end{problem}

If such a space $X_C$ exists, then we call it a ``\emph{$C$-generating space}", or simply, a ``\emph{generating space}". \\

By a \emph{cone} $C$ of a Banach space $X$ we mean that it is a convex set satisfying $C+C\subset C$. A cone $C$ is said to be  \emph{reproducing (resp., almost reproducing)} if $C-C=X$ (resp., $\overline{C-C}=X$, where $\overline{A}$ denotes the norm closure of $A\subset X$). An ordering cone $C$ of $X$ is a reproducing or an almost reproducing cone containing no nontrivial subspaces of $X$, or equivalently, satisfying $C\bigcap(-C)=\{0\}$. \\

This paper is organized as follows. In the first part of this paper (Sections 2-8), we are devoted to solving  I of Problem \ref{1.1}. As a result, we show that

\emph{(1) for every  nonempty closed convex $C$ of a separable Banach space $X$, the generating spaces $X_C$ always exist (Theorem \ref{3.6}), and they are not a unique in general;}

\emph{(2) if $C$ is a closed reproducing cone of a separable Banach space $X$, then two $C$-generating spaces $X_{C,p}$ and $X_{C,q}$ are isomorphic if and only if the two vectors $p,q$ are equivalent with respect to the order induced by the ``positive cone" $C$ (Theorem \ref{3.9});}

\emph{ (3) if  $X$ has an unconditional basis, in particular, $X=\ell_p$ ($1\leq p<\infty$), or, a separable $L_p$ ($1<p<\infty$), and $C$ is the positive cone with respect to the basis, then every generating space $X_C$ is isometric to $\ell_\infty$ (Theorem \ref{5.1});}

 \emph{(4) if $X$ is a separable Banach lattice and $C$ is the positive cone, then  every generating space $X_C$ is isometric to $L_\infty$.}\\

As applications of the results mentioned above, in the second part of this paper (Sections 9-11), we consider solvability of the following problems.\par
\begin{problem}\label{1.2}
Let $X,Y$ be Banach spaces, $Y$ be  ordered by an ordering cone $C\subset Y$ and   $f:X\rightarrow Y$ be a function. Assume that $\Omega\subset X$ is  a nonempty subset. Consider solvability of the following program.
\begin{equation}\label{vector model}
\begin{split}
\min\{f(x):\;x\in \Omega\}.
\end{split}
\end{equation}
A vector $\bar{x}\in \Omega$ is said to be a minimum point of the program (\ref{vector model}) provided
\begin{equation}\label{1.3}(f(\Omega)-f(\bar x))\bigcap (-C\setminus\{0\})=\emptyset.\end{equation}
\end{problem}
If such a vector $\bar x$ exists, then it is said to be a classical solution of the program (\ref{vector model}).  In order to obtain a classical solution of (\ref{vector model}), one usually assume that the ordering cone $C$ has nonempty interior, that is, int$C\neq\emptyset$. However, for an infinite dimensional Banach space $Y$, it is often impossible to claim that the assumption is true unless that $Y$ is a $C(K)$-like space. On the other hand, the assumption int$C\neq\emptyset$ is very important.  For example, when int$C=\emptyset$, any scalarization function of the program (\ref{vector model}) will deny continuity, while just the continuity guarantees that the Moreau-Rockafellar Theorem, i.e., Lemma \ref{Moreau-Rockafellar theorem} can be effectively  used to solvability discussion of the program (\ref{vector model}). These facts have forced optimization theorists to consider a kind of ``weak solution" $\bar{x}\in \Omega$ of the program (\ref{vector model}), namely, \emph{generalized approximate solutions}. That is, they substitute a subset $A$ of $Y$ containing the ordering cone $C$ but with int$A\neq \emptyset$ for $C$ so that
 \begin{equation}\label{1.3}(f(\Omega)-f(\bar{x}))\bigcap (-A\setminus\{0\})=\emptyset.\end{equation}
See, for example, \cite{10Durea,21Amahroq,10Truong,10Kasimbeyli}.
Nevertheless,  a solution $\bar x$ of (\ref{1.3}) is often not the classical one which optimization theorists are  most concerned.\par
\emph{In Section 9, we will apply results established in Sections 2-4 to generalize \emph{the exact penalty principle} due to J.J. Ye \cite{12Ye} to all separable Banach spaces. In Section 11,
  with the help of the results presented in Sections 2-8 and 10,  we will overcome the discontinuity difficulty of scalarization functions of the program (\ref{vector model}) in \emph{box constraints}   without any additional assumptions.}

\begin{problem}[Lagrange model with box constraint]\label{problem 1.4}
\begin{equation}\label{lagrange model}
\begin{split}
& \min\quad f(x)\\
& s. t. \quad g(x)\in -Y^+,\\
& \quad \quad\; h(x)=0_Z,\\
& \quad \quad\; x\in \Omega=\{x\in X:\;x_a\leq x\leq x_b\},\\
\end{split}
\end{equation}
\end{problem}
\noindent
where $X,Y,Z$ are Banach spaces and $X,Y$ are ordered by their corresponding ordering cones $X^+$, $Y^+$, and $f: X\rightarrow \mathbb{R}$ is a continuous convex function, $g:X\rightarrow Y$ is continuous and convex-like with respect to $Y^+$ and $h:X\rightarrow Z$ is a continuous affine  function. $\Omega=\{x\in X:\;x_a\leq x\leq x_b\}$ is called ``box constraint". We should mention that box constraint is a very common condition in applications, and that the program (\ref{lagrange model}) is a very general model which  contains many concrete models as its special cases. See, for instance, \cite{18Causa,07Daniele,11Donato,16Donato,10Durea,14Maugeri,06Rosch} and references therein.\par

In applications, one often considers the following Lagrange duality program.
\begin{problem}[Lagrange duality model]\label{problem 1.5}
\begin{equation}\label{duality program of constraint scalar program}
\begin{split}
&\max\limits_{(y^*,z^*,x_1^*,x_2^*)}\inf\limits_{x\in X} f(x)+\left<(y^*,x_1^*,x_2^*),\;(g(x),x-x_a,x_b-x)\right>+\left<z^*,\;h(x)\right>\\
&s.t.\quad x\in X,\;y^*\in Y^{*+},\;x_1^*\in X^{*+},\;x_2^*\in X^{*+},\;z^*\in Z^*.
\end{split}
\end{equation}
If extremal values of the programs (\ref{lagrange model}) and  (\ref{duality program of constraint scalar program}) are equal, then we say the duality between the two programs holds.
\end{problem}
Slater's condition is an important tool in obtaining the duality between  (\ref{lagrange model}) and  (\ref{duality program of constraint scalar program}). However, the duality may not hold in the case that int$(Y^+)=\emptyset$, because Slater's condition requires int$(Y^+)\neq\emptyset$.  It is worth to mention that, Daniele, Giuffr\`{e}, and Idone \cite{07Daniele} in their remarkable work  gave a condition called \emph{assumption (S)} which assure the Lagrange duality by substituting the nonsupport point set $(Y^+)_N$ of $Y^+$ for the interior int$(Y^+)$. Taking that as a starting point, a number of  mathematicians presented some reasonable  conditions and assumptions in various optimization problems (see, for instance, \cite{11Donato,16Donato,13Bazan,14Maugeri,14Maugeri2} and  references theirin).  It is worth to mention that in a few special cases, some applications are successful. (See, for example, Donato \cite{16Donato}).  Nevertheless,  the mentioned conditions and assumptions are usually hard to verify in practice, although some of these works contain great theoretical elegance \cite{08Bot,15Zalinescu,16Martin,15Zalinescu}.\par

\emph{To overcome the difficulty that ${\rm int}(Y^+)=\emptyset$ in Slater's condition and Lagrange duality theory, in Section 10, we will use the results presented in Sections 2-8 to transform program (\ref{lagrange model}) with box constraints to another equivalent program which Slater's condition can be applied directly. In Section 11, we consider the convex vector  optimization program (\ref{vector model}). As a result, we present some subdifferential inclusion theorems of Gerstewitz scalarization functions (Theorems \ref{new fermat rule} and \ref{subdifferential of composition});  and as their application, solvability of the vector variational inequality problem (\ref{variational inequality}) is discussed.  } \\
\;

In this paper, unless stated explicitly otherwise, we always assume that $X$ is a Banach space, and $X^*$ its dual. For a subset $A\subset X$, we denote successively, by $\overline{A}$, ${\rm co}A$ and $\overline{\rm co}A$, the (norm) closure, the convex hull and the closed convex hull of $A$.\par

\section{Preliminaries}
In this section, we will recall some concepts and basic properties related to the subset of  nonsupport points of a closed convex set in a Banach space $X$.
\begin{definition}\label{2.1} Let $C$ be a closed convex set of a Banach space $X$.

i) A point $x\in C$ is said to be a support point of $C$ provided there exists a non-zero functional $x^*\in X^*$ such that
\[\langle x^*,\;x\rangle=\max\{\langle x^*,\;y\rangle:\;y\in C\}.\]
In this case, $x^*$ is called a support functional of $C$, which is supporting $C$ at $x$.
We denote by $C_S$ the set of all support points of $C$.\par

ii) A point of the complement $C_N\equiv C\setminus C_S$  is called a non-support point of $C$, and $C_N$ is said to be the quasi-interior of $C$.\par

iii) We say that a  point $x$ of $C$ is a proper support point if there is a functional $x^*\in X^*$ such that
 \[\max\{\langle x^*,\;y\rangle:\;y\in C\}=\langle x^*,\;x\rangle>\inf\{\langle x^*,\;y\rangle:\;y\in C\}.\]
 We denote by $C_{PS}$ the set of all proper support points of $C$.\par

iv) A  point $x$ of $C$ is said to be a non-proper support point if it is not a proper support point of $C$. We use $C_{NP}$ to denote the set of all non-proper support points of $C$. Clearly, \[C_{NP}=C\setminus{C_{PS}}.\]

\end{definition}
\begin{definition}\label{2.1'} Let $C$ be a  convex set of a Banach space $X$.

i) $C$ is said to be a cone with vertex at the origin $0$ provided it is closed under the operations of addition and positive scalar multiplication, i.e.,
\[x,y\in C,\;\lambda\geq0\;\;{\rm imply\;} x+y\in C\;\;{\rm and\;}\;\lambda x\in C.\]\par

ii) We say that $C$ is a cone with vertex at  $x_0\in X$ if $C-x_0$ is a cone with vertex at the origin $0$.
When the vertex is origin $0$, we also say $C$ is pointed. \par

iii) A cone $C$ is called a reproducing  (resp., an almost reproducing) cone  if $C-C=X$ (resp., $\overline{C-C}=X$).

\end{definition}
The next property is easy to observe.
\begin{proposition}
Let $C$ be a closed cone of a Banach space $X$ with $C_N\neq\emptyset$. Then $C$ is almost reproducing. But the converse version is not true. (See, Examples \ref{2.4} and \ref{2.5}.)
\end{proposition}
\begin{example}
i) Let $(\Omega,\sum,\mu)$  be a measure space and $1\leq p\leq\infty$. Then  the positive cone $L_p^+(\mu)$  of $L_p(\mu)$ is reproducing. $(L_p^+(\mu))_N\neq\emptyset$ if and only if
$(\Omega,\sum,\mu)$  is $\sigma$-finite.

ii) For any nonempty set $\Gamma$ and $1\leq p\leq\infty$, the positive cone $\ell_p^+(\Gamma)$  of $\ell_p(\Gamma)$  is reproducing.

iii) For every topological space $K$, the positive cone $C^+(K)$ of the real-valued bounded continuous function space $C(K)$ endowed with the sup-norm is reproducing.

iv) For any nonempty set $\Gamma$, the positive cone $c_0^+(\Gamma)$  of  $c_0(\Gamma)$  is reproducing.
\end{example}

 For a closed convex $C\subset X$, and $x\in C$, let $C_x$ be the cone generated by $C$ with vertex $0$ which is defined by
 \begin{equation}\label{2.2} C_x=\bigcup_{\lambda>0}\lambda(C-x).\end{equation}

Each property in the following lemma is either easy to observe, or, to be found in Phelps \cite{72Phelps}, Cheng and Dong \cite{CD}, and Holmes \cite{75Homles}.
\begin{lemma}\label{2.3}
Suppose that $C$ is a nonempty closed convex set in a Banach space $X$. Then\par
(i)  $C_N$ is a convex subset of $C$ (maybe empty);\par
(ii) \cite[Ex. 2.18]{75Homles} if $X$ is separable, then $C_N=\emptyset$ if and only if $C$ is contained in a closed hyperplane;\par
(iii) $C_N={\rm int}C$ if the latter is nonempty;\par
(iv)  \cite{72Phelps} $x\in C_N$ if and only if $C_x$ is dense in $X$;\par
(v)  \cite{72Phelps} if $C_N\neq\emptyset$, then $C_N$ is a dense $G_\delta$-subset of $C$, hence, Baire's category;\par
(vi) \cite[Prop. 2.2]{CD} if $C$ is separable, then $C_{NP}\neq\emptyset$.
\end{lemma}

Now, we give several examples related to the nonsupport point sets of the natural positive cones of some classical Banach spaces as follows.

\begin{example}\label{2.4}
Let $(\Omega,\sum,\mu)$ be a measure space, and $L_p^+(\mu)$ be the positive cone of $L_p(\mu)$ for $1\leq p\leq\infty$.

i)   $(L_\infty^+(\mu))_N=\{f\in L_\infty(\mu): f(\omega)>0\; {\rm for\;a. e.\;}\omega\in \Omega\}$$={\rm int}(L_\infty^+(\mu));$

ii) $(L_p^+(\mu))_N={\rm int}(L_p^+(\mu))\neq\emptyset$  ($1\leq p<\infty$) if and only if $L_p(\mu)$ is finite dimensional;

iii)  $(L_p^+(\mu))_N=\{f\in L_p(\mu): f(\omega)>0\; {\rm for\;a.e.\;}\omega\in\Omega\}\neq\emptyset$\; ($1\leq p<\infty$) if and only if $(\Omega,\sum,\mu)$ is $\sigma$-finite.
\end{example}

\begin{example}\label{2.5} Let $\Gamma$ be a nonempty set and
$\ell_p^+(\Gamma)$ (resp., $c_0^+(\Gamma)$)  be the positive cone of $\ell_p(\Gamma)$ for $1\leq p\leq\infty$ (resp., $c_0(\Gamma)$). Then

i)   $(\ell_\infty^+(\Gamma))_N=\{x\in \ell_\infty: x(\gamma)>0\; {\rm for\;all\;\;}\gamma\in\Gamma\}={\rm int}(\ell_p^+(\Gamma));$

ii) $(\ell_\infty^+(\Gamma)_N={\rm int}(L_p^+(\Gamma))\neq\emptyset$  ($1\leq p<\infty$) if and only if $\Gamma$ is a finite set;

iii)  $(\ell_p^+(\Gamma))_N=\{x\in \ell_p(\Gamma): x(\omega)>0\; {\rm for\;all\;}\gamma\in\Gamma\}\neq\emptyset$\; ($1\leq p<\infty$) if and only if $\Gamma$ is countable.

iv)  $(c_0^+(\Gamma))_N=\{x\in c_0(\Gamma): x(\gamma)>0\; {\rm for\;all\;\;}\gamma\in\Gamma\}={\rm int}(c_0^+(\Gamma))\neq\emptyset$ if and only if $\Gamma$ is a finite set;

v) $(c_0^+(\Gamma))_N=\{x\in c_0(\Gamma): x(\gamma)>0\; {\rm for\;all\;\;}\gamma\in\Gamma\}\neq\emptyset$ if and only if $\Gamma$ is countable.
\end{example}
\begin{example}\label{2.6} Let $K$ be a compact Hausdorff space, and $C^+(K)$ be the positive cone of $C(K)$. Then
\[(C^+(K))_N=\{f\in C(K): f(k)>0\;{\rm for\;all\;}k\in K\}={\rm int}(C^+(K)).\]
\end{example}
\begin{definition}\label{2.7}
Assume that $C\subset X^*$ is a cone with vertex at the origin $0$. Then

i) $x^*\in X^*$ is said to be a positive functional with respect to $C$ provided $\langle x^*,x\rangle\geq0$ for all $x\in C$. Without causing confusion, we call it a positive functional for short.

ii) A functional $x^*\in X^*$ is strictly positive if $\langle x^*,x\rangle>0$ for all $x\in C\setminus\{0\}$.

iii) We denote by $C^{*+}$ the cone of all positive functionals (with respect to $C$).
\end{definition}
\begin{definition}\label{2.8}
Let $C\subset X$ be a convex set with $0\in C$. Then the Minkowski functional $p: X\rightarrow\mathbb R^+\cup\{+\infty\}$ generated by $C$ is defined for $x\in X$ by
\[p(x)=\inf\{\lambda>0: \lambda^{-1} x\in  C\}.\]
\end{definition}
The following property easily follows.
\begin{proposition}\label{2.9}
Let $C\subset X$ be a convex set with $0\in C$, and $p: X\rightarrow\mathbb R^+\cup\{+\infty\}$ be the Minkowski functional generated by $C$.
Then

i) $p$ is a nonnegative extended-real-valued sublinear function on $X$;

ii) $p$ is lower semicontinuous if and only if $C$ is closed in $X$; or, equivalently, $C'=\{x\in X: p(x)\leq1\}$ is closed;

iii) $p$ is continuous if and only if $0\in{\rm int}C$.
\end{proposition}
Following theorem is from \cite[Lemma 4.1]{16Cheng}.
\begin{theorem} Suppose that $C$ is a closed almost reproducing cone with vertex at the origin of a Banach space $X$. Then $x^*\in X^*$
is a strictly positive functional if and only if $x^*$ is a non-support point of $C^{*+}$.
\end{theorem}
At the end of this section, we recall some definitions and lemmas in convex analysis.
\begin{definition}\label{definition of subdifferential}
Let $X$ be a Banach space, $f:X\rightarrow \overline{\mathbb{R}}$ be a convex function. Suppose $f$ is lower semicontinuous at $\bar{x}\in dom f$. Then the subdifferential of $f$ at $\bar{x}$, denoted by $\partial f(\bar{x})$ is defined as
\[\partial f(\bar{x})=\{x^*\in X^*:\;\left<x^*,x-\bar{x}\right>\leq f(x)-f(\bar{x})\}.\]
The normal cone of a convex set $\Omega\subset X$ at $x\in \Omega$ is given by
\[N(x,\Omega)=\{x^*\in X^*:\;\left<x^*, y-x \right>\leq 0,\;\forall\;y\in \Omega\}.\]
If $x\notin \Omega$, we put $N(x,\Omega)=\emptyset$. Then we see
\[\partial f(\bar{x})=\{x^*\in X^*:\;(x^*,-1)\in N((\bar{x},f(\bar{x})),epif)\},\]
Recall $\delta(x,\Omega)=0$ if $x\in \Omega$; $\delta(x,\Omega)=\infty$, otherwise. We have
\begin{equation}\label{delta and N}
\partial \delta(x,\Omega)=N(x,\Omega).
\end{equation}
\end{definition}

Next we introduce the Moreau-Rockafellar Theorem.
\begin{theorem}\label{Moreau-Rockafellar theorem}
Suppose that $f,g:X\rightarrow \overline{\mathbb{R}}$  are convex proper lower semicontinuous functions on Banach space $X$ and that there is a point $x$ in $dom f \bigcap dom g$ where $f$ is continuous at $x$. Then
\[\partial(f+g)(x)=\partial f(x)+\partial g(x),\;x\in dom(f+g).\]
\end{theorem}

\begin{definition}\label{2.12}
Let $X,Y$ be Banach spaces, $Y$ be ordered by $Y^+$. Suppose $f:X\rightarrow Y$, $f$ is said to be convex-like with respect to $Y^+$ in $\Omega\subset X$ if the set $\{f(x)+y:\;y\in Y^+,\;x\in \Omega\}$ is convex.\par

$f$ is said to be convex respect to $Y^+$ on $\Omega$ if
\[\lambda f(x)+(1-\lambda)f(y)\leq f(\lambda x+(1-\lambda)y)\]
($\leq$ is in $Y^+$ sense) holds for any $x,y\in \Omega,\; 0\leq \lambda \leq 1$.\par
\end{definition}

\section{$C$-generating spaces}
\begin{definition}
 Assume that $C$ is a convex subset  of a real Banach space $X$.

 i) We denote by $[C]$ the closure of span$C$ in $X$, and by $C_{N,[C]}$ the set of all nonsupport points of $C$ with respect to $[C]$;

ii)  For any fixed $x\in C$, we say that the
subspace $C_x\bigcap (-C_x)$ is the $(C,x)$-generating space. If it causes no confusion, we call it a $C$-generating space for short,
where $C_{x}=\bigcup_{\lambda>0}\lambda(C-x)$.
\end{definition}

\begin{proposition}\label{3.1}
Assume that $C$ is a nonempty closed convex set of a finite dimensional normed space $X$. Then
\[C_N={\rm int}C.\]
\end{proposition}
\begin{proof}
Clearly, $C_N={\rm int}C$ if ${\rm int}C\neq\emptyset.$ Assume dim$X=n\in\mathbb N.$ Suppose, to the contrary, that ${\rm int}C=\emptyset,$ and that there is $x\in C_N$. Then by Lemma \ref{2.3} iv), \[C_{x}=\bigcup_{\lambda>0}\lambda(C-x)=\bigcup_{n=1}^\infty n(C-x)\] is a dense convex set in $X$. Therefore, $C_x$ contains $n$ linearly independent vectors $x_1, x_2,\cdots,x_n$ so that $\pm x_1, \pm x_2,\cdots,\pm x_n\in C_x$. Since ${\rm co}\{\pm x_1, \pm x_2,\cdots,\pm x_n\}$ is a symmetric convex body containing the origin in its interior, $C_x=X$. Since each $n(C-x)$ is closed in $X$, it follows from completeness of $X$ and Baire's category theorem  that int$[n(C-x)]\neq\emptyset$. Consequently, int$C\neq\emptyset$, and this is a contradiction. \qed
\end{proof}

\begin{lemma}\label{3.2}
Suppose that $C\subset X$ is a nonempty convex set. Then 
\begin{equation}\label{3.3}
C_{NP}=C_{N,[C]}.
\end{equation}
\end{lemma}
\begin{proof}
 By definition of $ X_C$ and Lemma \ref{2.3} iv), $x\in C_{N,[C]}$ if and only if $C_x=\bigcup_{\lambda>0}\lambda(C-x)$ is dense in $[C]$. If $x\notin C_{NP}$,
then  $x\in C_{PS}$. Therefore, $x\in C_{PS}\bigcap C_{N,[C]}=\emptyset$, which is a contradiction. \qed
\end{proof}
\begin{lemma}\label{3.4}
Suppose $C$ is a closed convex set of a Banach space $X$, and that  $x\in C_{NP}$.
Then the $(C,x)$-generating space satisfies
\begin{equation}\label{3.5} C_x\bigcap(-C_x)=\bigcup_{\lambda>0}\lambda\big((C-x)\bigcap(-C+x)\big), \end{equation}
and  $C_x\bigcap(-C_x)$ is a dense subspace of $[C]$.
\end{lemma}
\begin{proof} Without loss of generality, we assume that $[C]=X$.  We first show (\ref{3.5}). Clearly,
\begin{align*}\label{3.5} C_x\bigcap(-C_x)&=\big(\bigcup_{\lambda>0}\lambda(C-x)\big)\bigcap\big(\bigcup_{\lambda<0}\lambda(C-x)\big)\\
&\supset\bigcup_{\lambda>0}\lambda\big((C-x)\bigcap(-C+x)\big). \end{align*}
On the other hand, note that \begin{align*}0\neq z\in C_x\bigcap(-C_x)&\Longleftrightarrow \exists \;\lambda_1, \lambda_2>0,\;c_1, c_2\in C\;{\rm so\;that\;}\\
&\lambda_1(c_1-x)=\lambda_2(-c_2+x).\end{align*}
Equivalently, $z$ is absorbed by both $C-x$ and $x-C$, which is equivalent to that  $z$ is absorbed by  $(C-x)\bigcap(x-C)$. Consequently, $z\in \bigcup_{\lambda>0}\lambda\big((C-x)\bigcap(-C+x)\big)$.
Therefore, (\ref{3.5}) holds.

To show that  $C_x\bigcap(-C_x)$ is a dense subspace of $X$, it suffices to prove that $C_x\bigcap(-C_x)$ is  dense. Otherwise, it is contained in a closed hyperplane containing the origin.
Let $0\neq x^*\in X^*$ be such that \[C_x\bigcap(-C_x)\subset H(x^*;0)\equiv\{z\in X: \langle x^*,z\rangle=0\}.\]
Thus, for each $z\in C_x$ with $\langle x^*,z\rangle>0$, $-z\notin C_x$. Consequently, \[\max\{\langle -x^*,z\rangle: z\in C_x\}=-\min\{\langle x^*,z\rangle: z\in C_x\}=0=\langle -x^*,0\rangle,\] which says that $0$ is a support point of $C_x$. This is a contradiction. \qed
\end{proof}

\begin{theorem}\label{3.6}
Suppose that $C\subset X$ is a closed bounded convex set, and $e\in C_N$. Let $X_e$ be the $(C,e)$-generating space $C_e\bigcap(-C_e)$ endowed with the norm $\|\cdot\|_e$ defined for $x\in X_e$ as
\[\|x\|_e=\inf\{\lambda>0: \lambda^{-1}x\in (C-e)\bigcap(e-C)\}.\]
Then

i) Regarding as a subspace of $X$, $X_e$ is dense in $X$;

ii) The new norm $\|\cdot\|_e$-topology on $X_e$ is stronger than the original norm $\|\cdot\|$-topology on $X_e$;

iii) $X_e=(X_e,\|\cdot\|_e)$ is a Banach space;

iv)  $(X_e,\|\cdot\|_e)$ is isomorphic to  $(X_e,\|\cdot\|)$ if and only if $e\in{\rm int}C$.
\end{theorem}
\begin{proof}
i) This is just Lemma \ref{3.4}.

ii) \& iii)\;  Note that closed unit ball $(C-e)\bigcap(e-C)$ of $(X_e,\|\cdot\|_e)$ is  $\|\cdot\|$-closed in $X$. Then it is necessarily $\|\cdot\|$-complete.
Since $(C-e)\bigcap(e-C)$ is bounded symmetric convex absorbing set of $X_e$, and since $\|\cdot\|_e$ is the Minkowski functional generated by $(C-e)\bigcap(e-C)$,
the new norm $\|\cdot\|_e$-topology is not weaker than the original norm topology on $X_e$. Therefore, ii) is shown.  Completeness of $\big((C-e)\bigcap(e-C),\|\cdot\|\big)$ entails that
 $\big((C-e)\bigcap(e-C),\|\cdot\|_e\big)$ is complete. Consequently, $X_e=(X_e,\|\cdot\|_e)$ is a Banach space.  Hence, iii) is true.

iv) \;By ii),  $\|\cdot\|_e$-topology  is stronger than  $\|\cdot\|$-topology on $X_e$. This and iii) imply that  $(X_e,\|\cdot\|_e)$ is isomorphic to  $(X_e,\|\cdot\|)$ if and only if
$(X_e,\|\cdot\|)$ is also a Banach space. It follows from  i) that $(X_e,\|\cdot\|)$ is  a Banach space if and only if $X_e=X$, which is equivalent to that $0\in{\rm int}_{\|\cdot\|}(C-e)\bigcap(e-C)$, that is, $e\in{\rm int}_{\|\cdot\|}C$. \qed
\end{proof}
\begin{theorem}\label{3.5'}
Let $C$ be a closed bounded convex set of a Banach space $X$ containing at least two points. If $C$ is separable, then

i) $C_{N,[C]}\neq\emptyset$;

ii) for every $e\in C_{N,[C]}$, the $(C,e)$-generating space $(X_e,\|\cdot\|_e)$ is linearly isometric to a closed subspace of $\ell_\infty$.
\end{theorem}
\begin{proof}
i)  This is just Lemma \ref{2.3} ii).

ii)  Without loss of generality, we assume $[C]=X$. Therefore, $X$ is separable and every  subset $A$ of $X^*$ is $w^*$-separable, i.e., $A$ is separable in in the weak-star topology $w^*$ of $X^*$.  By Theorem \ref{3.6}, $\|\cdot\|_e$ is stronger that $\|\cdot\|$ on $X_e$  and $X_e$ is $\|\cdot\|$-dense in $X$. Therefore, $X^*\subset X^*_e$.

Since $(C-e)\bigcap(e-C)$ is a closed bounded convex subset of $X$, $f\equiv\|\cdot\|_e$ acting as an extended real-valued  Minkowski functional generated by $(C-e)\bigcap(e-C)$ defined on $X$ is lower semicontinuous with its effective domain \[{\rm dom}(f)=X_e=\bigcup_{n=1}^\infty n((C-e)\bigcap(e-C)).\]
Note that
\[x^*\in\partial f(x)\Longleftrightarrow \langle x^*,z\rangle \leq f(z)\;{\rm for\;all}\;z\in X\;{\rm with\;} \langle x^*,x\rangle=\|x\|_e,\]
and that
\[x^*\in\partial \|x\|_e\Longleftrightarrow \langle x^*,z\rangle \leq f(z)\;{\rm for\;all}\;z\in X_e\;{\rm with\;} \langle x^*,x\rangle=\|x\|_e.\]
Then we see \begin{equation}\label{3.6'}\partial f(x)\subset\partial\|x\|_e\subset S_{X^*_e},\;\;{\rm for\;all\;}x\in X_e\setminus\{0\},\end{equation}
where $S_{X^*_e}$ is the unit sphere of $X^*_e$.

We denote by $D=\bigcup_{x\in X\setminus\{0\}}\partial f(x)\subset X^*$, the range of the subdifferential mapping $\partial f$ of $F$. Then by the Brondsted-Rockafellar theorem \cite[Theorem 3.17]{ph},
we obtain \[f(x)=\sup_{x^*\in D}\langle x^*,x\rangle,\;\;{\rm for\;all\;}x\in X.\]
Note the restriction of $f$ to $X_e$ is just $\|\cdot\|_e$.
 $w^*$-separability of $D$ entails that there is a $w^*$-dense sequence $\{x_n^*\}$ of $D$ so that
 \[f(x)=\sup_{n\in \mathbb N}\langle x_n^*,x\rangle=\|x\|_e,\;\;{\rm for\;all\;}x\in X_e.\]
 Finally, we define $T: X_e\rightarrow\ell_\infty$ by
 \[Tx=(\langle x^*_1,x\rangle, \langle x^*_2,x\rangle,\cdots,\langle x^*_n,x\rangle,\cdots),\;x\in X_e. \]
 Clearly, $T: X_e\rightarrow\ell_\infty$ is a linear isometry. \qed
\end{proof}

Let $P$ be a reproducing cone containing no nontrivial subspaces with vertex at the origin of a real Banach space $X$.  Then there is an order on $X$ induced by $P$:
\begin{equation}\label{3.7} x\geq y\;\Longleftrightarrow x-y\in P,\;\forall x,y\in X. \end{equation}
If, in addition, $P$ satisfies that $P\bigcap{-P}=\{0\}$, then
\begin{equation}\label{3.8}
x\in X,\;\exists\; y\in P\;{\rm such \;that}\;\;-ty\leq x\leq ty,\;\forall t>0\;\Longrightarrow x=0.
\end{equation}
The following theorem states that in particular, if $C$ is an almost reproducing cone with vertex at the origin containing no nontrivial  subspaces, then the generating space can  be induced the ``cone order".

\begin{theorem}\label{3.9} Suppose that $C\subset X$ is a closed almost reproducing cone with vertex at the origin containing no nontrivial  subspaces of $X$,  $X$ is ordered by $C$ defined as (\ref{3.7}), and that $C_N\neq\emptyset$.
Let $u\in C_N$, $X_u=C_u\bigcap (-C_u)$, and   $\|\cdot\|_u$ is defined for $x\in X_u$ by
\begin{equation}\label{3.10}\|x\|_u=\inf\{\lambda>0:-\lambda u\leq x\leq\lambda u\}\equiv\inf\{\lambda>0:x\in \lambda[-u,u]\}.\end{equation}
Then\par

i) Regarding as a subspace of $X$, $X_u$ is dense in $X$;\par
ii) $\|\cdot\|_u$ is a norm on $X_u$, and the norm topology generated by it  is stronger than the topology generated  by $\|\cdot\|$ on $X_u$;\par
iii) The closed unit ball $[-u,u]$ of $(X_u,\|\cdot\|_u)$ is just $(C-u)\bigcap(u-C)$;\par
iv) The new norm $\|.\|_u$ is a lower semicontinuous function in $X$ with its essential domain dom$\|\cdot\|_u=X_u$;\par
v)  $(X_u,\|\cdot\|_u)$ is a Banach space.
\end{theorem}

\begin{proof}
i) It follows immediately from  Lemma \ref{3.4}.\par
ii) We first show that $[-u,u]$ is an absorbing set of $X_u$. Given $x\in X_u$, by (\ref{3.5}) of Lemma \ref{3.4}, there exist $\lambda>0,\;c_j\in C,\;j=1,2$ such that
\[-\lambda u\leq\lambda(c_1-u)=x=\lambda(u-c_2)\leq\lambda u.\]
Hence, $x\in\lambda[-u,u]$, and this says that $[-u,u]$ is an absorbing set of $X_u$. By (\ref{3.8}), it is easy to observe that $[-u,u]$ is bounded closed in $X$.
Since $[-u,u]$ is convex and symmetric, $\|\cdot\|_u$ is a norm on $X_u$, and the new norm topology is stronger than the topology generated  by $\|\cdot\|$ on $X_u$.
Therefore, ii) has been shown.\par
iii) It suffices to note that \begin{align*}x\in (C-u)\bigcap(u-C)&\Longleftrightarrow\;\exists\;c_1,c_2\in C\;{\rm such\;that}\;c_1-u=x=u-c_2\\
&\Longleftrightarrow\;-u\leq x\leq u\Longleftrightarrow\;x\in[-u,u].
\end{align*}

iv) Note $\bigcup_{\lambda>0}\lambda[-u,u]=X_u$. It follows from that  $[-u,u]$ is a closed bounded symmetrically convex set and that $\|\cdot\|_u$ is just the Minkowski functional generated by $[-u,u]$.

v) It follows from that  $[-u,u]$ is complete in $(X,\|\cdot\|)$ and ii). \qed
\end{proof}
\begin{theorem}[Equivalence theorem]\label{3.11}
 Suppose that $C\subset X$ is a closed reproducing cone  containing no nontrivial affine subspaces of $X$,  $X$ is ordered by $C$ defined as (\ref{3.7}), and that $C_N\neq\emptyset$.
Let $u,v\in C_N$,  $X_u=C_u\bigcap (-C_u)$, $X_v=C_v\bigcap (-C_v)$, $\|\cdot\|_u$  and   $\|\cdot\|_v$ are defined  by
\[\|x\|_u=\inf\{\lambda>0:x\in \lambda[-u,u]\},\;x\in X_u,\]
and
\[\|x\|_v=\inf\{\lambda>0:x\in \lambda[-v,v]\},\;x\in X_v.\]
Then the following statements are equivalent.\par

i) $X_u=X_v$ algebraically;

ii) $u\sim v$, i.e., there is a constant $c\geq1$ such that  $c^{-1}u\leq v\leq cu$;

iii) $\|\cdot\|_u\sim\|\cdot\|_v$, i.e., there is a constant $c\geq1$ such that  \[c^{-1}\|\cdot\|_u\leq\|\cdot\|_v\leq c\|\cdot\|_u.\]
\end{theorem}
\begin{proof}

i) $\Longrightarrow$ iii). Suppose that $X_u=X_v$. By Theorem \ref{3.9}, the closed unit ball $[-u,u]$ of $(X_u,\|\cdot\|_u)$ is $\|\cdot\|_v$-complete. Indeed, since $[-u,u]$ is $\|\cdot\|$-complete, and since the norm $\|\cdot\|_v$-topology is stronger than the original norm $\|\cdot\|$-topology on $X_v (=X_u)$, $[-u,u]$ is necessarily $\|\cdot\|_v$-complete.
Note that $X_v=\bigcup n[-u,u]$ endowed with $\|\cdot\|_v$ is a Banach space. Then by Baire's category theorem, $0\in{\rm int}_{\|\cdot\|_v}[-u,u]$. Therefore, there is a constant $a>0$
such that $[-v,v]\subset a[-u,u]$, or, equivalently, $\|\cdot\|_v\geq a^{-1}\|\cdot\|_u$. We can show that there is a constant $b>0$ such that  $\|\cdot\|_u\geq b^{-1}\|\cdot\|_u$ in the same way. Thus, we finish the proof of ``i) $\Longrightarrow$ iii)" by taking $c=\max\{a,b,1\}$.

iii) $\Longrightarrow$ i). It follows directly from  \[c^{-1}\|\cdot\|_u\leq\|\cdot\|_v\leq c\|\cdot\|_u\]
that $X_u=X_v$.

ii) $\Longleftrightarrow$ iii). It suffices to note that ii) $c^{-1}u\leq v\leq cu$ if and only if
\[c^{-1}[-u,u]\subset[-v,v]\subset c[-u,u],\]
which is equivalent to iii)
\[c^{-1}\|\cdot\|_u\leq\|\cdot\|_v\leq c\|\cdot\|_u. \qed\]
\end{proof}
\section{Examples of $C$-generating spaces}

For a closed convex set $C$ of a Banach space $X$ with $C_N\neq\emptyset$, and for each $e\in C_N$, by Theorem \ref{3.6}, $(X_e,\|\cdot\|_e)$ is a Banach space but  $(X_e,\|\cdot\|)$ is a dense subspace of $X$. This means that $X_e$ is algebraically smaller. The following examples will show that, usually, $(X_e,\|\cdot\|_e)$ is  bigger topologically, unless ${\rm int}C\neq\emptyset$.

\begin{example}\label{4.1} Let $(\Omega,\sum,\mu)$ be a probability space, i.e., $\mu(\Omega)=1$. In particular, $\Omega=[0,1]$, $\sum$ is the Borel $\sigma$-algebra  of $[0,1]$, and $\mu$ is the Lebesgue measure. Given $1\leq p<\infty$, we consider $X=L_p(\mu)$. Then it is easy to see that the constant function $u=1$ a nonsupport point of the positive cone $L_p^+(\mu)$.
Since \[[-u,u]=\{f\in L_p(\mu): -1\leq f(\omega)\leq1\;\;{\rm for\;almost\;all\;}\omega\in\Omega\},\]
Therefore, $X_u=L_\infty(\mu)$.\\
\;
If $p=\infty$, then $u=1\in{\rm int}(L_\infty^+(\mu))$ is just the ``generating unit element" of $L_\infty(\mu)$. Therefore, $X_u=X$.
\end{example}

\begin{example}\label{4.1'}
We consider $X=\ell_1$. Let $\{e_n\}$ be the standard unit vector basis of $\ell_1$ and $u=(\frac{1}{2^j})_{j=1}^\infty$. Then $X_u\cong\ell_\infty$. Indeed, it is clear that $u\in (\ell_1^+)_N$ with $\|u\|=1$ and
\[[-u,u]=\{x\in\ell_1:-\frac{1}{2^n}\leq x(n)\leq\frac{1}{2^n}\;\;{\rm for\;all\;}n\in\mathbb N\}.\]

The set ${\rm ext}[-u,u]$ of all extreme points of $[-u,u]$ satisfies
\[{\rm ext}[-u,u]=\big\{(\pm 2^{-1},\pm 2^{-2},\cdots,\pm 2^{-n},\cdots)\big\}.\]
Therefore,
\[\|(\pm 2^{-1},\pm 2^{-2},\cdots,\pm 2^{-n},\cdots)\|_u=1.\]
It also follows that $\|\cdot\|_u$ is monotone non-decreasing, i.e., for every pair of  sequences $\{a_j\}_{j=1}^\infty, \{b_j\}_{j=1}^\infty\subset\mathbb R$
\[\mid a_j\mid \geq \mid b_j\mid , j=1,2,\cdots,n \;\Longrightarrow \|\sum_{j=1}^na_je_j\|_u\geq\|\sum_{j=1}^nb_je_j\|_u,\]
and
\[\|e_n\|_u=2^n,\;n=1,2,\cdots.\]
We defined a linear operator $T: X_u\rightarrow\ell_\infty$ for $x=(x(n))_{n=1}^\infty\in X_u$ by
\[T(x)=(2x(1),2^2x(2),\cdots,2^nx(n),\cdots).\]
Then \[T(\pm 2^{-1},\pm 2^{-2},\cdots,\pm 2^{-n},\cdots)=(\pm1,\pm1,\cdots,\pm1,\cdots).\]
Therefore,
\[T\big({\rm ext}[-u,u]\big)={\rm ext}B_{\ell_\infty}.\]
Consequently,
\[T(B_{X_u})=B_{\ell_\infty},\]
This says that $T: X_u\rightarrow\ell_\infty$ is a surjective linear isometry.
\end{example}

\section{A congruence theorem }
In this section, we focus on generating spaces of Banach spaces with unconditional bases.
\begin{lemma}\label{5.1} Let $X$ be a Banach space with a normalized 1-unconditional basis $\{e_n\}$, and $X^+$ be the positive cone of $X$ related to the basis, i.e.,
\[X^+=\{x=\sum_na_ne_n\in X: a_n\geq0\;{\rm for\;all\;}n\in\mathbb N\}.\]
Then

i) $(X^+)_N=\{x=\sum_na_ne_n\in X: a_n>0\;{\rm for\;all\;}n\in\mathbb N\}$.

ii) For all $u\in (X^+)_N$, $(X_u,\|\cdot\|_u)$ is isometrically isomorphic to $\ell_\infty$.

\end{lemma}
\begin{proof} We write \[X_{00}=\big\{x=(x(n))\in X,\;{\rm supp}x=\{n\in\mathbb N, x(n)\neq0\}\;{\rm is\;a\;finite\;set}\big\}.\]
Then it is a dense subspace of $X$.

i) Suppose that $u=\sum_na_ne_n\in X$ with $a_n>0$ for all $n\in\mathbb N$. Then for any $x\in X_{00}$, Let $m=\max\{n\in{\rm supp}x\}$, and let $\alpha=\min\{a_n:n=1,2,\cdots,m\}$. Then there is $M>0$ such that
\[\mid x(j)\mid \leq M\alpha, \;j=1,2,\cdots,m.\]
Therefore, $x\in C_u$. Since $x\in X_{00}$ is arbitrary, it follows $X_{00}\subset C_u$. This and Lemma \ref{2.3} iv) imply that $u\in (X^+)_N$.  Conversely, suppose that $u=\sum_na_ne_n\in X^+$ with $a_j=0$ for some $j\in\mathbb N$. Let $e_j^*$ be the functional such that $\ker{e_j}^*=\overline{\rm span}\{e_1,e_2,\cdots,e_{j-1},e_j,\cdots\}$ and $\langle e_j^*,e_j\rangle=-1.$
Then \[0=\langle e_j^*,u\rangle=\max\{\langle e_j^*,z\rangle: z\in X^+\}. \]
Therefore, $u$ is a support point of  $X^+$.

ii) Let $u=\sum_na_ne_n\in X^+$ with $a_n>0$ for all $n\in\mathbb N$. Then
\[[-u,u]=\{x\in X: -a_n\leq x(n)\leq a_n\},\]
and
\[{\rm ext}[-u,u]=\{(\varepsilon_1 {a_1},\varepsilon_2 {a_2},\cdots,\varepsilon_n {a_n},\cdots):\varepsilon_j\in\{-1,1\},j=1,2,\cdots\}.\]
Let $T: X_u\rightarrow\ell_\infty$ for $x=(x(n))_{n=1}^\infty\in X_u$ be defined by
\[T(x)=({a_1}^{-1}x(1), {a_2}^{-1}x(2),\cdots,{a_n}^{-1}x(n),\cdots).\]
Then \[T(\varepsilon_1 {a_1},\varepsilon_2 {a_2},\cdots,\varepsilon_n {a_n},\cdots)=(\varepsilon_1,\varepsilon_2,\cdots,\varepsilon_n.\cdots),\]
$\varepsilon_j\in\{-1,1\},j=1,2,\cdots.$
Therefore,
\[T\big({\rm ext}[-u,u]\big)={\rm ext}B_{\ell_\infty}.\]
Consequently,
\[T(B_{X_u})=B_{\ell_\infty},\]
This says that $T: X_u\rightarrow\ell_\infty$ is a surjective linear isometry. \qed
\end{proof}

As a consequence of Lemma \ref{5.1}, we have the following result.
\begin{corollary}\label{5.2}
Let $1\leq p\leq\infty$, $\ell_p^+$ be the natural positive cone of $\ell_p$  and $u\in(\ell_p^+)_N$. Then
\[(X_u,\|\cdot\|_u)\cong\ell_\infty.\]
\end{corollary}
\begin{proof}
Since the natural unit vector basis $\{e_n\}$ of $\ell_p$ ($1\leq p<\infty$) is  an unconditional 1-basis of $\ell_p$, and since the natural positive cone  $\ell^+_p$ of  $\ell_p$ is just the positive cone of $\ell_p$ related to the basis $\{e_n\}$, it follows from Theorem \ref{5.1} whenever $1\leq p<\infty$.

If $p=\infty$, then  \[({\ell^+_\infty})_N={\rm int}{\ell^+_\infty}=\{x=(x(n)): 0<\alpha\equiv\inf_n x(n)\leq\sup_n x(n)\equiv\beta<\infty\}.\]
Therefore, for any fixed $u=(u(n))\in({\ell^+_\infty})_N$ we have
\[[-u,u]=\{x\in\ell_\infty: -u(n)\leq x(n)\leq u(n),\;\forall\;n\in\mathbb N\},\]
and
\[{\rm ext}[-u,u]=\big\{(\varepsilon_1 {u(1)},\varepsilon_2 u(2),\cdots,\varepsilon_n {u(n)},\cdots): \varepsilon_j\in\{-1,1\}\big\}.\]
Consequently,
$T: X_u\rightarrow\ell_\infty$ defined  for $x=(x(n))_{n=1}^\infty\in X_u$  by
\[T(x)=({u(1)}^{-1}x(1), {u(2)}^{-1}x(2),\cdots,{u(n)}^{-1}x(n),\cdots)\]
is a linear surjective isometry. \qed
\end{proof}
\begin{theorem}\label{5.3}
Let $X$ be a Banach space with an unconditional basis $\{e_n\}$, and $X^+$ be the positive cone of $X$ associated with $\{e_n\}$, i.e.,
\[X^+=\{x=\sum_na_ne_n\in X: a_n\geq0\;{\rm for\;all\;}n\in\mathbb N\}.\]
Then

i) $(X^+)_N=\{x=\sum_na_ne_n\in X: a_n>0\;{\rm for\;all\;}n\in\mathbb N\}$.

ii) For all $u\in (X^+)_N$, $(X_u,\|\cdot\|_u)$ is isometrically isomorphic to $\ell_\infty$.
\end{theorem}
\begin{proof}
Without loss of generality, we can assume that $\{e_n\}$ is a 1-unconditional basis of $X$. Otherwise, let $f_n=e_n/\|e_n\|,\;n=1,2,\cdots$, and
 let $\parallel \mid \cdot \parallel \mid$ be defined for $x\in X$ by
\[\parallel \mid x \parallel \mid=\sup_{N\in\mathbb N,\varepsilon_n\in\{-1,1\}}\|\sum_{n=1}^N\varepsilon_n a_nf_n\|,\;\;x=\sum_{n=1}^\infty a_nf_n.\]
Then $\{f_n\}$ is a normalized 1-unconditional basis of $(X, \parallel \mid \cdot \parallel \mid)$.\par
 Note $(X,\parallel \mid \cdot \parallel \mid)^+=(X,\parallel \mid \cdot \parallel \mid)^+$. Then by Lemma \ref{5.1}, for each $u\in (X^+)_N$, the  Banach space
$X_u=(X_u,\parallel \cdot \parallel_u)$ generated by the closed convex set $[-u,u]$ is isometric to $\ell_\infty$. \qed
\end{proof}
Recall that the Haar system, i.e. the sequence of functions $\{\chi_{n}(t)\}_{n=1}^\infty$ defined on the interval $[0,1]$  by $\chi_1(t)\equiv1$, and for $k=0,1,2,\cdots,$ $j=1,2,\cdots,2^k$,
\begin{equation}\chi_{2^k+j}(t)=\left\{\begin{array}{ccc}
                    1~, & \;\;\;\;\;\;\;\;{\rm if\;}t\in[(2j-2)2^{-k-1},(2j-1)2^{-k-1}]; \\
                    -1~,& {\rm if\;}t\in((2j-1)2^{-k-1},2j2^{-k-1}];\\
                    0~,& {\rm otherwise}\;\;\;\;\;\;\;\;\;\;\;\;\;\;\;\;\;\;\;\;\;\;\;\;\;\;\;\;\;\;\;\;
                  \end{array}
\right.
\end{equation}
is (in the given order) a monotone (but not normalized) unconditional basis of $L_p[0,1]$ ($1<p<\infty$) with basis constant at most $p^*-1$, where $q^*=\max\{p,q\},$ and $\frac{1}p+\frac{1}q=1$ (See,  \cite[Theorem 6.1.7]{alb}). That is, \[\|\sum_{n=1}^N\varepsilon_na_n\chi_n\|\leq (p^*-1)\|\sum_{n=1}^Na_n\chi_n\|,\]
for every real number sequence $\{a_n\}_{n=1}^\infty$, $N\in\mathbb N$ and $\varepsilon_n\in\{-1,1\}$.
Then we obtain the following result.
\begin{corollary}\label{5.2}
Let  $1<p<\infty$, and $C$ be the positive cone of $L_p[0,1]$ with respect to the Haar basis $\{\chi_n(t)\}_{n=1}^\infty$, i.e.
\[C=\{x=\sum_na_n\chi_n\in X: a_n\geq0\;{\rm for\;all\;}n\in\mathbb N\}.\]
 Then for every $u\in C_N$,
\[(X_u,\|\cdot\|_u)\cong\ell_\infty.\]
\end{corollary}
\begin{proof}
Since $\{\chi_n(t)\}_{n=1}^\infty$ is an unconditional basis of $L_p[0,1]$,  it follows from Theorem \ref{5.1} that for all $u\in C_N$, $(X_u,\|\cdot\|_u)$ is isometrically isomorphic to $\ell_\infty$. \qed
\end{proof}

\begin{corollary}\label{5.3}
Let   $c^+_0$ be the positive cone of $c_0$.  Then for every $u\in (c^+_0)_N$,
\[(X_u,\|\cdot\|_u)\cong\ell_\infty.\]
\end{corollary}
\begin{proof} By Theorem \ref{5.1}, it suffices to note that $c^+_0$ is just the positive cone of $c_0$ with respect to the standard unit vector basis $\{e_n\}$ of $c_0$. \qed
\end{proof}

\section{More on $L_p$ spaces}
\begin{lemma}\label{6.1} For any $u\in (L^+_1[0,1])_N$, or, $(L^+_\infty[0,1])_N$, we have $X_u\cong L_\infty[0,1]$.
\end{lemma}
\begin{proof}
i) Let  $u\in (L^+_1[0,1])_N$. Then $u(t)>0$ a.e. for $t\in[0,1]$.
let $T: X_u\rightarrow L_\infty[0,1]$ be defined for $f\in X_u$ by
$Tf=u^{-1}f$. Then we obtain that
\[T[-u,u]=[-1,1]=B_{L_\infty[0,1]}.\]
Therefore, $X_u\cong L_\infty[0,1]$.

ii) Note that
\[(L^+_\infty[0,1])_N={\rm int}(L^+_\infty[0,1])=\{f\in L^+_\infty[0,1]: 0<{\rm essinf}f\}. \]
For every $u\in(L^+_\infty[0,1])_N$, let $T: X_u\rightarrow L_\infty[0,1]$ be defined for $f\in X_u$ by
$Tf=u^{-1}f$. Then $T$ is a bounded linear operator from $L_\infty[0,1]$ to itself. Clearly,
\[T[-u,u]=[-1,1]=B_{\ell_\infty}.\]
Therefore, $T$ is a surjective isometry from $L_\infty[0,1]$ to itself. That is,  $X_u\cong L_\infty[0,1]$. \qed
\end{proof}

\begin{theorem}\label{6.2}
Let $(\Omega,\sum,\mu)$ be a $\sigma$-finite measure space, and $1\leq p\leq\infty$. Then

i) $(L^+_p(\mu))_N\neq\emptyset$;

ii) for each $u\in(L^+_p(\mu))_N\neq\emptyset$, the generating space $(X_u,\|\cdot\|_u)\cong L_\infty(\mu).$
\end{theorem}
\begin{proof}
i) Assume $1\leq p<\infty$. Then it is easy to observe that
\[(L^+_p(\mu))_N=\{f\in L_p(\mu): f(\omega)>0\;a.e. \;\omega\in\Omega\}.\]
Therefore, $(L^+_p(\mu))_N\neq\emptyset$. If $p=\infty$, then $(L^+_\infty(\mu))_N={\rm int}L^+_\infty(\mu)\neq\emptyset$.

ii) Given $u\in (L^+_p(\mu))_N$, let $T: X_u\rightarrow L_\infty(\mu)$ be defined  by
\[Tf=u^{-1}f, f\in X_u.\]
Then we obtain that $T$ is a bounded linear operator and satisfies
\[T[-u,u]=[-1,1].\]
Consequently, $T:X_u\rightarrow L_\infty(\mu)$ is a linear surjective isometry. \qed
\end{proof}
\begin{theorem}\label{6.3}
Let $(\Omega,\sum,\mu)$ be a $\sigma$-finite measure space, and $1<p<\infty$. Then there is a closed reproducing cone $C$ of $L_p(\mu)$ with $C_N\neq\emptyset$ such that
for each $u\in C_N$, we have the generating space $X_u\cong\ell_\infty$.
\end{theorem}
\begin{proof}
Let $\Omega_1, \Omega_2\in\sum$ satisfy that $\Omega=\Omega_1\bigcup\Omega_2$, $\Omega_1\bigcap\Omega_2=\emptyset$ and that $(\Omega_1,\sum_1,\mu)$ is atomless and $(\Omega_2,\sum_2,\mu)$
is atomic, where $\sum_j=\Omega_j\bigcap\sum,\;j=1,2$. Since $(\Omega,\sum,\mu)$ is $\sigma$-finite,  $\Omega_2$ is countable. Therefore, \[L_p(\Omega,\sum,\mu)=L_p(\Omega_1,{\sum}_1,\mu)\oplus_p L_p(\Omega_2,{\sum}_2,\mu)\;\;\;\;\;\;\;\;\]
\begin{equation}\nonumber
\cong\left\{\begin{array}{cc}
                    L_p(\Omega_1,\sum_1,\mu)\oplus_p \ell_p, & \;\;\;\;\;\;\;\;{\rm if\;}\Omega_2\;{\rm is\;\; infinite;\;\;\;\;\;\;\;\;\;\;} \\
                    L_p(\Omega_1,\sum_1,\mu)\oplus_p \ell^n_p~,& \;\;\;\;\;{\rm if\;} \Omega_2\;{\rm has \;n\; elements}.
                  \end{array}
\right.
\end{equation}

Case I. $\mu(\Omega_1)=0$. It is trivial.

Case II. $0<\mu(\Omega_1)<\infty$. Since $L_p(\Omega_1,\sum_1,\mu)\cong L_p[0,1]$, it follows from Corollary \ref{5.2}.

Case III. $\mu(\Omega_1)=\infty$. Since $(\Omega_1,\sum_1,\mu)$ is $\sigma$-finite, there is a $\sigma$-partition $\{E_j\}$ of $\Omega_1$ such that $0<\mu(E_j)<\infty$ for all $j\in\mathbb N$. Therefore, $L_p(E_j,\mu)\cong L_p[0,1]$ for all $j\in\mathbb N$. Consequently,
 \[L_p(\Omega_1,{\sum}_1,\mu)=\bigoplus_j L_p(E_j,\mu)\cong\bigoplus_j L_p([0,1])\cong L_p([0,1]).\]
 It again follows from Corollary \ref{5.2}. \qed

\end{proof}

\section{On lattice versions}
In this section, we will again discuss $C$-generating spaces in the particular case that $X$ is a Banach lattice and $C$ is the ``positive" cone of $X$. For more information concerning lattice theory, we refer the reader to \cite{85Aliprantis,06Aliprantis,74Schaefer}.

 All notions related to Banach lattices and abstract $M$ spaces are the same as in Lindenstrauss and Tzafriri \cite{lin}.  Recall that a partially ordered real Banach space $Z$ is called a Banach lattice provided

(i) $x\leq y$ implies $x+z\leq y+z$ for all $x,y,z\in Z$;

(ii) $ax\geq0,$ for all $x\geq0$ in $Z$ and $a\in\mathbb R^+$;

(iii) both $x\vee y$ and $x\wedge y$ exist for all $x,y\in Z$;

(iv) $\|x\|\leq\|y\|$ whenever $\mid x\mid \equiv x\vee{-x}\leq y\vee{-y}\equiv \mid y\mid$.

It follows from iv) and
\begin{equation}\label{7.1}
\mid-y\mid=\mid x\vee z-y\vee z\mid+\mid x\wedge z-y\wedge z\mid, \;\;{\rm for\;all\;}x,y,z\in Z,
\end{equation}
that the lattice operations are norm continuous.

By a sublattice of a Banach lattice $Z$ we mean a linear subspace $Y$ of $Z$ so that $x\vee y$ (and also $x\wedge y=x+y-x\vee y$) belongs to $Y$ whenever $x,y\in Y$. A lattice ideal $Y$ in $Z$ is a sublattice of $Z$ satisfying that $\mid z\mid \leq \mid y\mid$ for $z\in Z$ and for some $y\in Y$ implies $z\in Y$.

A Banach lattice $X$ is said to be an abstract $M$  space ($AM$-space, for short) if
\[x,y\in X,\;\;\mid x\mid \wedge \mid y\mid =0\Longrightarrow \|x+y\|=\max\{\|x\|,\|y\|\}.\]
A mapping $T$ from a partially ordered real Banach space $X$ to a  partially ordered real Banach space $Y$ is said to be an order isometry provided it is a fully order preserving isometry, i.e. $\|Tx-Ty\|=\|x-y\|$ for all $x,y\in X$ and $Tx\geq Ty$ if and only if $x\geq y$.
\begin{lemma}\label{7.2}
Let $X$ be a Banach lattice and $X^+=\{\mid x\mid: x\in X\}$ be the positive cone. Then

i) for every $u\in(X^+)_N$,
$X_u\equiv(X_u,\|\cdot\|_u)$ is again a Banach lattice and $u$ is its unit.

ii) If, in addition, $X$ is a  function space consisting of real-valued functions defined on a set $\omega$, and the lattice operations are induced by its natural order, i.e., $x\geq y$ if and only if $x(\omega)\geq y(\omega)$ for all $\omega\in \Omega$, then $(X_u,\|\cdot\|_u)$ is  an AM-space with the unit $u$.
\end{lemma}
\begin{proof}
i) Note that the closed unit ball $[-u,u]$ of $X_u$ is a bounded  closed symmetric convex subset of $X$ with respect to the norm topology of $X$. Clearly, $ax\geq0,$ for all $x\geq0$ in $X_u$ and $a\in\mathbb R^+$.
For every pair $x,y\in [-u,u]$, by properties of the lattice operations on $X$, we have
\[-u\leq x\wedge y\leq x\vee y\leq u.\]
This implies that both $x\vee y$ and $x\wedge y$ exist for all $x,y\in X_u$, and further,
$x\leq y$ implies $x+z\leq y+z$ for all $x,y,z\in X_u$. Let $x, y\in X_u$. Then $\mid y \mid\leq u$ and $\mid x\mid \leq \mid y\mid$ imply that  $-u\leq x\leq u$. This is clearly equivalent to that $\|y\|_u\leq 1$ and $\mid x\mid\leq \mid y\mid$ entail $\|x\|_u\leq1$. Thus, for every $x,y\in X_u$, if $\mid x\mid \leq \mid y\mid $, then $\|x\|_u\leq\|y\|_u$. Thus, $(X_u,\|\cdot\|_u)$ is again a Banach lattice with the unit $u$.

ii) Note that $\mid x\mid \wedge\mid y \mid=0$ is equivalent to that ${\rm supp}x\bigcap{\rm supp}y=\emptyset$ if $X$ is a function space.  Let $x,y\in X_u$ satisfy  $\mid x\mid \wedge\mid y\mid=0$.  Then \begin{equation}\label{7.3}\mid x+y\mid =\mid x\mid +\mid y\mid =\mid x\mid \vee \mid y\mid + \mid x\mid \wedge \mid y\mid =\mid x\mid \vee \mid y\mid.\end{equation} Since $(X_u,\|\cdot\|_u)$ is a Banach lattice, it follows from (\ref{7.3}) that  \begin{equation}\label{7.4}\|x+y\|_u\geq\|x\|_u\vee\|y\|_u.\end{equation} On the other hand, for every pair $x,y\in X_u$ with $\mid x\mid \wedge\mid y\mid =0$,
$x+y\in[-u,u]$ if and only if $x,y\in[-u,u]$. Consequently,   $\|y\|_u\leq\|x\|_u=1$ and  $\mid x\mid \wedge\mid y\mid =0$ imply that \begin{equation}\label{7.5}\|x+y\|_u\leq1=\|x\|_u=\|x\|_u\vee\|y\|_u.\end{equation}
(\ref{7.4}) and (\ref{7.5}) together entail that for every pair $x,y\in X_u$ with $\mid x\mid \wedge\mid y\mid=0$ we have
\[\|x+y\|_u=\|x\|_u\vee\|y\|_u.\]
Therefore, $(X_u,\|\cdot\|_u)$ is  an AM-space with the unit $u$. \qed
\end{proof}
\begin{remark}
Lemma \ref{7.2} i) has been shown in \cite[Coro. p.102]{74Schaefer}. Here the proof is a different but simplified approach.
\end{remark}
Let $X$ be a Banach lattice with an order unit $e$, and ${X^*}^+$ be the positive cone of $X^*$, i.e. ${X^*}^+$ is the closed cone of $X^*$ consisting of all positive functionals with respect to the positive cone $X^+$ of $X$. We denote by $B_{X^*}^+=B_{X^*}\bigcap{X^*}^+$.
Let
\begin{align}\label{7.6}
K=\{x^*\in B_{X^*}^+: x^*\in{\rm ext}B_{X^*}\;{\rm with\;}\;\langle x^*,e\rangle=1\}.
\end{align}
Then $K$ is a compact Hausdorff space when it is endowed with the $w^*$-topology of $X^*$.

Keep these notions in mind. Then we state the  Kakutani-Bohnenblust-M. Krein-S. Krein Theorem as follows. The proof of  ``i) $\Longleftrightarrow$ iii)" can be seen in Bohnenblust and  Kakutani \cite{bo} and the proof of  ``i) $\Longleftrightarrow$ ii)"  is just \cite[Theorem 9.32]{06Aliprantis}.
\begin{theorem}[Kakutani-Bohnenblust-M. Krein-S. Krein]\label{7.7}
Let $X$ be a Banach lattice. Then the following statements are equivalent.

i) $X$ is a Banach lattice with a unit;

ii) $X$ is an AM-space with a unit;

iii) $X$ is order isometric to $C(K)$.  The space $K$ defined as (\ref{7.6}) is unique up to homeomorphism.
\end{theorem}

\begin{theorem}\label{7.8}
 Let $X$ be a Banach lattice so that its positive cone has nonempty interior. Then it is order isomorphic to a Banach space $C(K)$ for some compact Hausdorff space $K$, i.e., there is an equivalent lattice norm on $X$ so that   $X$ (with respect to the new norm)  is order isometric to  $C(K)$.
\end{theorem}
\begin{proof}
 Suppose that $X^+$ has nonempty interior. Then $(X^+)_N={\rm int}X^+$. By Lemma \ref{7.2}, for any $u\in {\rm int}X^+$,   the generating space $X_u$ is again a Banach lattice with the unit $u$. Therefore, it follows from Theorem \ref{7.7} that $X_u$ is order isometric to a space $C(K)$ for some compact Hausdorff space $K$.
By Theorem \ref{3.11}, $\|\cdot\|_u$ is an equivalent norm on $X$. \qed
\end{proof}
\begin{theorem}\label{7.9}
Let $X$ be a separable Banach lattice. Then

i) the set $(X^+)_N$ of nonsupport points of the positive cone  $X^+$ is nonempty;

ii) for every $u\in(X^+)_N$,
$X_u\equiv(X_u,\|\cdot\|_u)$ is order isometric to a Banach space $C(K)$ for some compact Hausdorff space $K$.
\end{theorem}
\begin{proof}
i) It follows from Lemma \ref{2.3}.

ii) It is a consequence of Lemma \ref{7.2} and Theorem \ref{7.7}. \qed
\end{proof}

\section{Latticization}

Let $X$ be a real Banach space, $\Omega=(B_{X^*},w^*)$, the closed unit ball $B_{X^*}$ of $X^*$ endowed with the weak-star topology $w^*$ of $X^*$. Let $E_{\mathscr K}$ be the closed subspace of $C(\Omega)$ consisting of all b-$w^*$-continuous positive homogenous functions, that is, all bounded weak-star continuous homogenous functions on $X^*$ but restricted to $\Omega$. It is shown in \cite{CCS} that $E_{\mathscr K}$ is a Banach lattice so that $X$ is order isometric to a subspace of $E_{\mathscr K}$. Therefore, we also call it the latticization of $X$.

We use $\mathscr K(X)$ (resp.,  $\mathscr K_0(X)$) to denote the cone of all nonempty convex compact subsets (resp., containing the origin $0$) of $X$ endowed with the Hausdorff metric $d_H$, i.e.
\begin{equation}\label{8.1}
d_H(A,B)=\inf\{r>0: A\subset B+rB_X,  B\subset A+rB_X\},\;A,B\in\mathscr K(X).
\end{equation}
Next, let $J:\mathscr K(X)\rightarrow E_{\mathscr K}$ be defined for $C\in\mathscr K(X)$ by
\begin{equation}\label{8.2}
J(C)(x^*)=\sup_{x\in C}\langle x^*,x\rangle,\;\;x^*\in\Omega.
\end{equation}
Let ``$\vee$" and ``$\wedge$"  be the usual lattice operations defined on $E_{\mathscr K}$, i.e., for all $u,v\in E_{\mathscr K}$ and $x^*\in\Omega$,
\[(u\vee v)(x^*)=\max\{u(x^*),v(x^*)\},\;\;(u\wedge v)(x^*)=\min\{u(x^*),v(x^*)\}.\]
Keep these notations in mind. Then we have the following result.

\begin{lemma}\label{8.3}
Let $X$ be a real Banach space. Then

i) $E_{\mathscr K}=\overline{J\mathscr K(X)-J\mathscr K(X)}$ is a Banach lattice;

ii) $J\mathscr K(X)-J\mathscr K(X)=J\mathscr K_0(X)-J\mathscr K_0(X)$;

iii)  $E_{\mathscr K}$ is separable if $X$ is separable.
\end{lemma}
\begin{proof}
i) This is just \cite[Theorem 3.2 i)]{CCS}.

ii) It is easy to check that $J\mathscr K(X)-J\mathscr K(X)=J\mathscr K_0(X)-J\mathscr K_0(X)$. Indeed, let $u=JC-JD$ for $C,D\in\mathscr K(X)$. Choose any $x_0\in C$, $y_0\in D$ and let
$A={\rm co}\{\pm x_0,\pm y_0\}$. Then $A+C,\;B+D\in\mathscr K_0(X)$ and $u=J(A+C)-J(A+D)$.

iii) Assume that $X$ is separable. Let $\{x_n\}\subset X$ be a dense subsequence of $X$, and for all $m\in\mathbb N$, let
\[\mathscr F_m=\big\{{\rm co}F: F\;{\rm is\;a\;subset\;of\;m\;elements\;of}\;\{x_n\}\big\},\]
and
\[\mathscr F=\bigcup_{n=1}^\infty\mathscr F_n.\]
Then $\mathscr F$ is a dense countable subset of $\mathscr K(X)$. Since $J:\mathscr K(X)\rightarrow J\mathscr K(X)\subset E_{\mathscr K}$ is an order isometry \cite[Theorem 2.3]{CCS},
$J\mathscr K(X)$ is separable in $E_{\mathscr K}$. Consequently, $E_{\mathscr K}$ is separable. \qed
\end{proof}

\begin{lemma}\label{8.4}
Let $X$ be a separable Banach space, and $Z\equiv E_{\mathscr K}$ be its latticization.  For any nonsupport point $u$ of the positive cone $Z^+=J\mathscr K_0(X)$ of $Z$, we obtain

i) $X\bigcap Z_u$ is a dense subspace of $X$ with respect to the norm of $X$;

ii) $X\bigcap Z_u$ is a dual space with respect to the new norm $\|\cdot\|_u$ of $Z_u$.

\end{lemma}
\begin{proof}
i) Let $C$ be a compact convex set of $X$ containing the origin such that $u=JC$. Since $u\in (Z^+)_N$, $Z_u=\bigcup_{n=1}^\infty n[-u,u]$ is a $\|\cdot\|$-dense subspace of $Z=E_\mathscr K$. This entails that $\bigcup_{n=1}^\infty n(C\cap-C)$  is a $\|\cdot\|$-dense subspace of $X$. Note $x\in X\bigcap Z_u$ if and only if there exists $\lambda>0$ such that $-JC=-u\leq\lambda x\leq u=JC$, which is equivalent to $\pm\lambda x\in C$. Therefore, $X\bigcap Z_u=\bigcup_{n=1}^\infty n(C\cap-C)$  is a $\|\cdot\|$-dense subspace of $X$.

ii) Since $X$ is a closed subspace of $Z=E_\mathscr K$, $X\bigcap Z_u$ is closed in $Z_u$. Note that  $x\in X\bigcap Z_u$ with $\|x\|_u\leq 1$ if and only if $x\in C\cap-C$, i.e. the closed unit ball of $X\bigcap Z_u$ is just $C\cap-C$.
Since $C\cap-C$ is compact in $X$, by the Dixier-Ng theorem (see, for instance, \cite[Theorem p.211]{75Homles}), $X\bigcap Z_u$ is a dual space with respect to the new norm $\|\cdot\|_u$ of $Z_u$. \qed
\end{proof}
\begin{theorem}\label{8.5}
Let $X$ be a separable Banach space, and $Z\equiv E_{\mathscr K}$ be its latticization.  For any nonsupport point $u$ of the positive cone $Z^+=J\mathscr K_0(X)$ of $Z$,

i) the space $Z_u$ is order isometric to $L_\infty(\mu)$ for some probability measure $\mu$;

ii) $X\bigcap Z_u$ is a closed subspace of $Z_u$ and it is also a dual space.
\end{theorem}
\begin{proof}
i) Since $Z$ is a separable Banach lattice, by Lemma \ref{7.2}, $Z_u$ is again a Banach lattice with its unit $u$. By Theorem \ref{7.8},
$Z_u$ is order isometric to  $C(K)$, where the Hausdorff space $K$ is defined as (\ref{7.6}). On the other hand, note that the closed unit ball of $Z_u$ is $[-u,u]$, where $u=JC$ for some compact convex set $C\subset X$ containing the origin. It is easy to see that $[-u,u]$ is a Lipschitz set in $C_b(\Omega)$, where $\Omega=(B_{X^*},w^*)$. Therefore,  $[-u,u]$ is compact in $C_b(\Omega)$. Again by the Dixier-Ng theorem, $Z_u$ is dual space. Consequently, there is an abstract $L$ space $Y$ so that $Y^*=Z_u$ (See, for example, \cite[Theorem 9.27]{06Aliprantis}).  By the Kakutani theorem, there is a probability measure $\mu$ such that $Y=L_1(\mu)$. Therefore,  $Z_u=L_1(\mu)^*=L_\infty(\mu)$. \qed

\end{proof}
\section{Exact penalization}
In this section, we will use properties of $C$-generating spaces established in the third section to generalize Ye's \emph{ exact penalty principle} \cite{12Ye}.
To begin with this section, we recall some definitions.
\begin{definition}\label{9.1}
Let $X,Y$ be Banach spaces, $Y$ be ordered by a (reproducing) cone $C$ of $Y$,  $S\subset X$ be a nonempty subset and $f: S\rightarrow Y$ be a mapping.

i) $f$ is said to be $C$-lipschitz on $S$ of rank $L_f$ with respect to $e$ if there exist a positive constant $L_f$ and an element $e\in C$ with $\|e\|=1$ such that
\[f(x)\leq f(y)+L_f\|x-y\|e,\forall x, y\in S.\]

 ii)  $f$ is called locally $C$-Lipschitz near $\bar{x}\in X$ with respect to $e$ if there is a neighborhood $U$ of $\bar{x}$ such that $f$ is $C$-Lipschitz  with respect to $e$  on $U$.\par
\end{definition}

The following problem is said to be constrained vector optimization problem.
\begin{equation}\label{original program}
\begin{split}
\min\quad\quad f(x)\\
{\rm s.t.}\quad x\in \Omega\subset S,\\
\end{split}
\end{equation}
where $X,Y$ are Banach spaces, $Y$ is ordered by a (reproducing) cone $C$ of $Y$,  $\Omega\subset S\subset X$ are  nonempty subsets and $f: S\rightarrow Y$ is a mapping.\\

The exact penalty approach aims at replacing a constrained optimization problem by an equivalent unconstrained optimization problem. Most results in the literature of exact penalization are mainly concerned with finding conditions under which a solution of the constrained optimization problem is a solution of an unconstrained penalized optimization problem, and the reverse property is rarely studied. In \cite{12Ye}, Jane J. Ye considered the reverse property. Precisely, she studied the following (unconstrained)  penalty program.

\begin{equation}\label{penalty program}
\begin{split}
\min\quad\quad f(x)+L_f d(x,\Omega)e\\
s.t.\quad\quad x\in S,\quad\quad\quad\\
\end{split}
\end{equation}
where $L_f\in\mathbb R^+$, $e\in C$ and  $d(x,\Omega)=\inf\{\|x-z\|:z\in \Omega\}$ is the distance function. She obtained the following \emph{global exact penalization for distance function} \cite[Theorem 3.1]{12Ye}, which extends \emph{Clarke's Exact Penalty Principle} from a scalar function to a vector function, and which states that the  constrained optimization problem (\ref{original program}) and the unconstrained exact penalized problem (\ref{penalty program}) under some conditions are exactly equivalent.

\begin{theorem}[Ye \cite{12Ye}]\label{9.2}
Let $X$ and $Y$ be Banach spaces, $Y$ be ordered by a convex (reproducing) cone $C$ of $Y$,  $S\subset X$ and $\Omega\subset S$ be nonempty subsets.
Let $f:S\rightarrow Y$ be $C$-Lipschitz on $S$ of rank $L_f$, and $e$ be an element in $C$ given by the $C$-Lipschitz continuity of $f$.

i) Assume  that $C\setminus\{0\}$ is an open set. Then any global minimizer of $f$ on $\Omega$ is a global minimizer of the exact penalty function $f(x)+L_f d(x,\Omega)e$ on $S$.

ii) Assume that either $S$ is closed or that $C\setminus\{0\}$ is an open set. Then, for any $L>L_f$, $x$ is a global $C$-minimizer of $f$ on $S$ if and
only if it is a global $C$-minimizer of the exact penalty function $f(x)+Ld(x, \Omega)e$ on $S$.\par
\end{theorem}

As we have mentioned in Section 1, and Theorem \ref{8.5}, a Banach space admitting a cone with nonempty interior is an almost a $C(K)$-space. Therefore,  \emph{Ye's Global Exact Penalization Theorem} above can be understood as a perfect extension of \emph{Clarke's exact penalty principle} from a scalar function to a  vector function valued in finite dimensional spaces. On the other hand, in infinite dimensional spaces, it is limited to the class of $C(K)$-spaces. Note that for a closed reproducing cone $K$ of a Banach space $Y$ with the set $K_N$ of all nonsupport points of $K$ being nonempty, in particular, $Y$ is separable, the set $C\equiv K_N\bigcup\{0\}$ is again an almost reproducing cone.  We  extend this principle in the following manner.


\begin{theorem}[Generalized global exact penalization for distance function]\label{9.3}
Let $X$ and $Y$ be  Banach spaces,  $Y$ be  ordered by the cone $C=K_N\bigcup\{0\}$ of a closed reproducing cone $K$ with $K_N\neq\emptyset$  of  $Y$,  $S\subset X$ and $\Omega\subset S$ be two nonempty subsets. Assume that $f:S\rightarrow Y$ is $C$-Lipschitz on $S$ of rank $L_f$, and that $e$ is the element in $C$ given by the $C$-Lipschitz continuity of $f$. Then we have the following assertions.

i) Every global $C$-minimizer of $f$ on $\Omega$ is a global $C$-minimizer of the exact penalty function $f(x)+L_f d(x,\Omega)e$ on $S$.\par

ii) For every $L>L_f$, a global $C$-minimizer  of the exact penalty function $f(x)+Ld(x, \Omega)e$ on $S$ if and only if it is a global $C$-minimizer of $f$ on $\Omega$.

\end{theorem}
\begin{proof}
 i)\; Suppose, to the contrary, that there exists a global minimizer $\bar{x}\in \Omega$ of the program (\ref{original program}) on $\Omega$ but it is not a global minimizer of the program (\ref{penalty program}) on $S$. Then there exists $\bar{z}\in S$ such that
\begin{equation}\label{9.0}
\begin{split}
f(\bar{z})+L_fd(\bar{z},\Omega)e<f(\bar{x}).
\end{split}
\end{equation}
Therefore, $f(\bar{x})-f(\bar{z})\in C\setminus\{0\}=K_N$. Since $f:S\rightarrow Y$ is $C$-Lipschitz of rank $L_f$ with respect $e$,
\[-L_f\|\bar{z}-\bar{x}\|e\leq f(\bar{x})-f(\bar{z})\leq L_f\|\bar{z}-\bar{x}\|e.\]
This and (\ref{9.0}) lead to
\begin{equation}\label{control eachother}
\begin{split}
L_fd(\bar{z},\Omega)e<f(\bar{x})-f(\bar{z})\leq L_f\|\bar{z}-\bar{x}\|e.
\end{split}
\end{equation}

Denote $q=f(\bar{x})-f(\bar{z})$. Then $q\in K_N$. Note $\overline{C}=K$. Then by Theorem \ref{3.9}, the following two $K$-generating spaces
\[Y_e=\bigcup\limits_{\lambda>0}\lambda(K_e\bigcap(-K_e)),\; Y_q=\bigcup\limits_{\lambda>0}\lambda(K_q\bigcap(-K_q))\]
are Banach spaces with respect to their new norms $\|\cdot\|_e$ and $\|\cdot\|_q$ defined by
 \begin{align*}
\|x\|_e&=\inf\{\lambda>0:x\in \lambda((C-e)\bigcap(e-C))\}\\
&=\inf\{\lambda>0:x\in [-e,e]\},\;\; x\in Y_e,
 \end{align*}
and
 \begin{align*}
\|x\|_q&=\inf\{\lambda>0:x\in \lambda((C-q)\bigcap(q-C))\}\\
&=\inf\{\lambda>0:x\in [-q,q]\},\;\; x\in Y_q.
 \end{align*}
It follows from (\ref{control eachother}) and Theorem \ref{3.11}, $Y_e=Y_q$ algebraically, and that $\|\cdot\|_e$, $\|\cdot\|_q$ are equivalent.
Clearly, $e,q\in{\rm int}{C_e}\bigcap{\rm int}{C_q}={\rm int}{C_e}$, where $C_e=\bigcup_{\lambda>0}\lambda(C-e)$. Consequently, there is $\delta>0$ so that
\begin{equation}\label{9.4}
f(\bar{x})-f(\bar{z})-\varepsilon e=q-\varepsilon e\in C_e,\;\;\;{\rm for\;all\;}\;0<\varepsilon<\delta.
\end{equation}
It follows from (\ref{control eachother}) and (\ref{9.4}) that for all sufficiently small $0<\varepsilon<\delta$,
\begin{equation}\label{crucial equation}
f(\bar{x})-f(\bar{z})-\varepsilon e>L_fd(\bar{z},\Omega)e.
\end{equation}
On the other hand,  there exists $x_\varepsilon\in \Omega$ such that
\[\|\bar{z}-x_\epsilon\|<\varepsilon+d(\bar{z},\Omega).\]
This and (\ref{crucial equation}) imply
\begin{equation}
\begin{split}
f(x_\varepsilon)\leq f(\bar{z})+L_f\|\bar{z}-x_\varepsilon\|e< f(\bar{z})+L_f(\varepsilon+d(\bar{z},\Omega))e<f(\bar{x}).
\end{split}
\end{equation}
This is a contradiction!\\

ii) Sufficiency.\; It  follows from i) we have just proven.

Necessity.\; Suppose, to the contrary, that there exists a global $C$-minimizer $\bar{x}\in S$ of the penalty function $f(x)+Ld(x,\Omega)e$ on $S$ but it is not a global $C$-minimizer of the program (\ref{original program}) on $\Omega$. If $\bar{x}\in \overline{\Omega}\bigcap S$, then there exists $\bar{z}\in \Omega$ such that
\begin{equation}\label{9.5}
\begin{split}
f(\bar{x})>f(\bar{z}).
\end{split}
\end{equation}
This and $d(\bar{x},\Omega)=0=d(\bar{z},\Omega)$ imply that
\begin{equation}\nonumber
\begin{split}
f(\bar{x})+Ld(\bar{x},\Omega)=f(\bar{x})>f(\bar{z})=f(\bar{z})+Ld(\bar{z},\Omega),
\end{split}
\end{equation}
which is a contradiction to that $\bar{x}$ is a  global $C$-minimizer of the penalty function  $f({x})+Ld({x},\Omega)$ on $S$. Therefore, $\bar{x}\notin \overline{\Omega}\bigcap S$.
It follows from (\ref{9.5}) that
\[0<q\equiv f(\bar{x})-f(\bar{z})\in C\setminus\{0\}=K_N.\] 
By Lipschitz continuity of $f$,  
\[q=f(\bar{x})-f(\bar{z})\leq L_f\|\bar{x}-\bar{z}\|e.\] Therefore, $[-q,q]\subset[-\beta e,\beta e]$, where $\beta=L_f\|\bar{x}-\bar{z}\|.$
Consequently,  $Y_q\subset Y_e$ algebraically, and $q\in{\rm int}C_e\bigcap K_N$.
Hence, there is $\delta>0$ so that
\begin{equation}\label{9.6}
f(\bar{x})-f(\bar{z})-\varepsilon e=q-\varepsilon e\in {\rm int}C_e\bigcap K_N,\;\;\;{\rm for\;all\;}\;0<\varepsilon<\delta.
\end{equation}
On the other hand,  for each such $\varepsilon>0$, there exists $x_\varepsilon\in \Omega$ such that
\[\|\bar{x}-x_\epsilon\|<\varepsilon+d(\bar{x},\Omega).\]
This, $d({x_\varepsilon},\Omega)=0$  and Lipschitz continuity of $f$ imply
 \begin{align*}
 f(x_\varepsilon)&+Ld(x_\varepsilon,\Omega)=f(x_\varepsilon)\\
&\leq f(\bar{x})+L_f\|x_\varepsilon-\bar{x}\|e\\
&< f(\bar{x})+L_f(\varepsilon+d(\bar{x},\Omega))e\\
&= f(\bar{x})+L_fd(\bar{x},\Omega)e+L_f\varepsilon e.
 \end{align*}
Note that $L>L_f$ and $d(\bar{x},\Omega)>0$. By (\ref{9.6}), we choose $0<\varepsilon<\delta$ so that \[L_f\varepsilon<(L-L_f)d(\bar{x},\Omega).\]
Then it follows
 \[f(x_\varepsilon)+Ld(x_\varepsilon,\Omega)e<f(\bar{x})+Ld(\bar{x},\Omega)e.\]
This is a contradiction to that  $\bar{x}$ is a  global $C$-minimizer of the penalty function  $f({x})+Ld({x},\Omega)e$ on $S$. \qed
\end{proof}
\begin{remark}
Parallel to  the \emph{generalized global exact penalization for distance function},  we can show a \emph{generalized local exact penalization for distance function} in the same way, which can be regarded as an extension of \emph{Ye's local exact penalization for distance function} \cite[Theorem 3.2]{12Ye}.
\end{remark}

\section{Convex scalar optimization}
In this section, we consider solvability of the box constraint of Lagrange model (Problem \ref{problem 1.4}) and the Lagrange duality model (Problem \ref{problem 1.5}) mentioned in Section 1. We have already known that every Banach space $X$ is contained in its latticization $E_{\mathscr K}$, and the density character of $E_{\mathscr K}$ is the same as that of $X$ (Lemma \ref{8.3}). Therefore, without loss of generality, we can assume that every Banach space $X$ in question is a subspace of a  Banach lattice. Thus, $X^+$ and  $\mid x\mid=x\vee-x$ for every $x\in X$ are meaningful.

 Now, we restate Problems \ref{problem 1.4} and \ref{problem 1.5} as follows.
 \begin{problem}[Lagrange model with box constraint]\label{10.1}
\begin{equation}\label{10.2}
\begin{split}
& \min\quad f(x)\\
& s. t. \quad g(x)\in -Y^+,\\
& \quad \quad\; h(x)=0_Z,\\
& \quad \quad\; x\in \Omega=\{x\in X:\;x_a\leq x\leq x_b\},\\
\end{split}
\end{equation}
\end{problem}
\noindent
where $X,Y,Z$ are Banach spaces and $X,Y$ are ordered by their corresponding ordering cones $X^+$, $Y^+$ with $(X^+)_N\neq\emptyset\neq(Y^+)_N$ , and $f: X\rightarrow \mathbb{R}\cup\{+\infty\}$ is a lower semicontinuous convex function with $[\frac{1}rx_a,rx_b]\subset{\rm dom}f$ for some $r>1$, $g:X\rightarrow Y$ is $Y^+$-Lipschitz on $\Omega$ of rank $L_g$  with respect to some $e\in(Y^+)_N$ with $\|e\|=1$ (Definition \ref{9.1})  and convex-like with respect to $Y^+$ (Definition \ref{2.12}) and $h:X\rightarrow Z$ is a continuous affine  function, and $x_a,x_b\in X$ with $x_b-x_a\in (X^+)_N$. Denote $S=\{x\in \Omega:\;g(x)\in -Y^+,h(x)=0_Z\}$.
\begin{proposition}\label{convex like proposition}
Suppose $g$ is convex-like respect to $Y^+$ and $g$ is $Y^+$-Lipschitz respect to $e\in (Y^+)_N$ in $\Omega$, then $g$ is also convex-like respect to $Y_e^+$ in $\Omega$.
\end{proposition}
\begin{proof}
Denote $A_1=\{f(x)+y:\;y\in Y^+,\;x\in \Omega\}$, $A_2=\{f(x)+y:\;y\in Y_e^+,\;x\in \Omega\}$.
For any $x_1,\;x_2\in \Omega$, $y_1,y_2\in Y_e^+$, any $\alpha\in [0,1]$,
since $g$ is convex-like respect to $Y^+$, there exist $x_\alpha\in \Omega$ and $y_\alpha\in Y^+$ such that \[g(x_\alpha)+y_\alpha=\alpha(g(x_1)+y_1)+(1-\alpha)(g(x_2)+y_2)\in A_1.\]
Because $g(x_1), g(x_2), g(x_\alpha), y_1, y_2$ are all in $Y_e$, we have $g(x_\alpha)+y_\alpha\in Y_e$ which implies $y_\alpha\in Y^+_e$ and $g(x_\alpha)+y_\alpha\in A_2$. Thus $g$ is also convex-like respect to $Y^+_e$ in $\Omega$.\qed \par
\end{proof}

\begin{problem}[Modified Lagrange duality model]\label{problem 10.3}
\begin{equation}\label{10.4}
\begin{split}
&\max\limits_{(y^*,z^*,x_1^*,x_2^*)}\inf\limits_{x\in X_\pi} f(x)+\left<(y^*,x_1^*,x_2^*),\;(g(x),x-x_a,x_b-x)\right>+\left<z^*,\;h(x)\right>\\
&s.t. \quad\quad y^*\in Y_{e'}^{*+},\;x_1^*,\;x_2^*\in X_\pi^{*+},\;z^*\in Z^*,
\end{split}
\end{equation}
where $X_\pi$ (resp.,  $Y_{e'}$ ) is the $(X^+,\pi)$ (resp., $(Y^+,e')$ )-generating space for some $\pi\in(X^+)_N$ (resp.,  $e'\in(Y^+)_N$)  and $X_{\pi}^{*+}$ (resp., $Y_{e'}^{*+}$) is the positive cone of its dual $X_{\pi}^*$ (resp., $Y_{e'}^*$).
\end{problem}

\begin{theorem}\label{10.5}[Modified Slater's condition]
If the primal program (\ref{10.2}) is solvable and satisfying that there is $\bar{x}\in S$ such that there exists $\lambda>0$, $-\lambda g(\bar{x})>e$ ($e$ is from Lipschitz property of $g$), $h(\bar{x})=0$, and such that $h(\Omega)$ contains a neighborhood of $0$, then the dual program (\ref{10.4}) is also solvable, that is, it admits an optimal solution, and the extremal values of the two programs are equal.
\end{theorem}
\begin{proof}
Without loss of generality, we can assume $x_a, x_b\in (X^+)_N$. Otherwise, we substitute successively, $x_c=x_b-x_a$ for $x_a$ and  $x_d=2(x_b-x_a)$ for  $x_b$;
$\tilde{f}(x)=f(x-x_b+2x_a)$ for $f(x)$;   $\tilde{g}(x)=g(x-x_b+2x_a)$ for $g(x)$; $\tilde{h}(x)=h(x-x_b+2x_a)$ for $h(x)$, and $\Omega'=\{x\in X: x_c\leq x\leq x_d\}$ for $\Omega$ in the programs (\ref{10.2})  and (\ref{10.4}).

Let $\pi=x_b$ and $(X_\pi,\|\cdot\|_\pi)$ be the $(X^+,\pi)$-generating space. Then there exists $r>1$ such that
\begin{equation}\label{10.5'}\Omega\subset[\frac{1}rx_a,rx_b]\subset{\rm int}_{\|\cdot\|_\pi}(X^+\bigcap X_\pi^+).\end{equation}
Since $f$ is real-valued and lower semicontinous convex on $[\frac{1}rx_a,rx_b]$, $f$ is  lower semicontinous convex on $[\frac{1}rx_a,rx_b]$ with respect to the norm $\|\cdot\|_\pi$. (\ref{10.5'}) implies that $f$ is continuous (hence, locally Lipschitz) on $(\frac{1}rx_a,rx_b)$. Therefore, $f$ is  locally Lipschitz on $\Omega$.

Let $\bar{x}\in S$ such that $g(\bar{x})\in -(Y^+)_N$, $h(\bar{x})=0$, and such that $h(\Omega)$ contains a neighborhood of $0$. Note that $e-g(\bar{x})>e$. Let $e'=\frac{e-g(\bar{x})}{\|e-g(\bar{x})\|}$,  and let $(Y_{e'},\|\cdot\|_{e'})$ be the $(Y^+,e')$-generating space. Then $-\lambda g(\bar{x})>e$ for some $\lambda>0$ implies
\[-g(\bar{x})\in{\rm int}_{\|\cdot\|_{e'}}(Y_{e'}^+).\] Since $\|e'\|=1$, we obtain $\|\cdot\|_{e'}\geq\|\cdot\|$ on $Y_{e'}$.
Since $g:X\rightarrow Y$ is $Y^+$-Lipschitz on $\Omega$ of rank $L_g$  with respect to  $e\in(Y^+)_N$, it is also $Y^+$-Lipschitz on $\Omega$ of rank $L'_g$  for some $L'_g>0$ with respect to  $e'\in(Y^+)_N$, that is,
\[-L'_g\|x-y\|e'\leq g(x)-g(y)\leq L'_g\|x-y\|e', \;\;{\rm for\;all\;}x,y\in \Omega.\]
Therefore,
\[-L'_g\|x-y\|_{e'}e'\leq g(x)-g(y)\leq L'_g\|x-y\|_{e'}e', \;\;{\rm for\;all\;}x,y\in \Omega.\]
Consequently,
\begin{equation}\label{10.5''}
\{g(x)-g(y): x,y\in X\}\subset Y_{e'}.
\end{equation}
Since $g$ is convex-like with respect to $Y^+$, that is,
$\{g(x)+y: y\in Y^+, x\in X\}$ is a convex set of $Y$.  This and (\ref{10.5''}) entail that $\{g(x)+y: y\in Y_{e'}^+, x\in X\}$ is a convex set of $Y_{e'}$ by Proposition \ref{convex like proposition}.

Now, the program (\ref{10.2}) can be equivalently translated into the following one.
\begin{equation}\label{10.2'}
\begin{split}
& \min\quad f(x)\\
& s. t. \quad g(x)\in -Y_{e'}^+,\\
& \quad \quad\; h(x)=0_Z,\\
& \quad \quad\; x\in \Omega=\{x\in X_\pi:\;x_a\leq x\leq x_b\}.\\
\end{split}
\end{equation}
By the hypothesis of the theorem, the program (\ref{10.2'}) is solvable.
We consider the  corresponding dual  program  (\ref{10.4}).
\begin{equation}\nonumber
\begin{split}
&\max\limits_{(y^*,x_1^*,x_2^*,z^*)}\inf\limits_{x\in X_\pi} f(x)+\left<(y^*,x_1^*,x_2^*),\;(g(x),x-x_a,x_b-x)\right>+\left<z^*,\;h(x)\right>\\
&s.t.\quad y^*\in Y_{e'}^{*+},\;x_1^*,\;x_2^*\in X_\pi^{*+},\;z^*\in Z^*.
\end{split}
\end{equation}
It satisfies that there exists
$\bar{x}\in S$ such that $g(\bar{x})\in -{\rm int}(Y_{e'}^+)$, $h(\bar{x})=0$, and $h(\Omega)$ contains a neighborhood of $0$. By Slater's condition (See, for instance, \cite[Theorem 2]{18Causa}, also, \cite[Theorem 5.3]{96Jahn}), the dual program is also solvable. \qed
\end{proof}
As an application of Theorem \ref{10.5}, we consider solvability of the following \emph{elastic plastic torsion problem}. The elastic plastic torsion problem goes back decades (see, for instance, \cite{49Mises,51Prager,66Ting}). Later, a number of mathematicians studied this problem ( see\cite{79Caffarelli,81Friedman,14Daniele,14Maugeri,14Giuffre,15Giuffre}). This problem could be formulated as follows.

\begin{example}\label{example 10.1}[Elastic plastic torsion problem]
Let $\Omega\subset \mathbb{R}^n$ be a nonempty bounded open  Lipschitz domain with its boundary $\partial\Omega$, $H_0^1(\Omega)$  be the Sobolev space defined by
\[H_0^1(\Omega)=\{u\in L_2(\Omega): u=0\; {\rm on}\;\partial\Omega,\; \nabla u\in L_2(\Omega)\},\] endowed with the norm
\[\|u\|\equiv \|u\|_{H_0^1(\Omega)}=\sqrt{\int_\Omega({\nabla u})^2d\omega}.\]
Let \[K\equiv \{u\in H_0^1(\Omega):\;u\geq 0,\;({\nabla u})^2=\sum\limits_{i=1}^{n}(\frac{\partial u}{\partial x_i})^2\leq 1\;a.e.\;in\;\Omega\},\]
and let \[Y=L_1(\Omega).\]
 Then for every $u\in H_0^1(\Omega)$, we have \[Y\equiv L_1(\Omega)\ni \sum\limits_{i=1}^{n}(\frac{\partial u}{\partial x_i})^2-1.\]
A vector $\bar u\in K$ is said to be a solution of the elastic plastic torsion problem if it is a minimum to following convex program:
\begin{equation}\label{elastic plastic torison problem}
\begin{split}
& \min\quad f(\bar u,u)\equiv a(\bar u,u)+\int_\Omega gu dx\\
& s.t.\quad u\in K
\end{split}
\end{equation}
where  $a:\;H_0^1(\Omega)\times H_0^1(\Omega)\rightarrow \mathbb{R}$ is a bilinear function, and $g\in L^2(\Omega)$.\par

Let $X\equiv H_0^1(\Omega)$. Then $X$ is a separable Hilbert space.
Since $X$ is separable, by Lemma \ref{2.3}, $(X^+)_N\neq\emptyset$. Indeed,
\begin{equation}\label{10.6}
(X^+)_N=\{u\in X: u>0\;\;a.e. \;\Omega\}.
\end{equation}
  Note that $K$ is a bounded absolutely convex set of the (Hilbert) space $X$. It is easy to check that $0\in K_N$ ( the set of nonsupport points of $K$).
Let $X_\pi$ be the $(K,0)$-generating space, that is $X_\pi=\bigcup_{n=1}^\infty nK$ endowed with the norm $\|\cdot\|_\pi$ which is the Minkowski functional generated by $K$. Then by Theorem \ref{3.5'}, $X_\pi$ is linearly isometric to a closed subspace of $\ell_\infty$. Since $\ell_\infty$ is linearly isomorphic to $L_\infty(\Omega)$, $X_\pi$ can be regarded as a closed subspace of $L_\infty(\Omega)$.


Let $Y\equiv L_1(\Omega)$, $e=\frac{1}{\mu(\Omega)}\in Y$, i.e. the constant function on $\Omega$, where $\mu(\Omega)$ is the Lebesgue measure of $\Omega$. Then
$e\in (Y^+)_N$ with $\|e\|=\int_\Omega ed\mu=1$.
By Example \ref{4.1}, the $(Y^+,e)$-generating space $Y_e=L_\infty(\Omega)$.
Therefore,  the operator $\sum\limits_{i=1}^{n}(\frac{\partial }{\partial x_i})^2-1$ can be regarded as a continuous operator from $X_\pi$ to $Y_e$. Indeed,  $\sum\limits_{i=1}^{n}(\frac{\partial }{\partial x_i})^2-1$ is trivially from $X_\pi$ to $Y_e$. To show continuity of $\sum\limits_{i=1}^{n}(\frac{\partial }{\partial x_i})^2-1$, it suffices to note $\sum\limits_{i=1}^{n}(\frac{\partial u }{\partial x_i})^2-1=(\nabla u)^2-1$, and $\nabla: X_\pi\rightarrow Y_e=L_\infty(\Omega)$ is a bounded linear operator.\par

Since $\Omega$ is an open set, for every fixed $x_0\in \Omega$, there exist  $r>0$ such that $\{x:\;\|x-x_0\|\leq r\}\subset \Omega$. Fix $0<\lambda<\frac{1}{2r}$ and let
\begin{equation}u(t)=\left\{\begin{array}{ccc}
                   \frac{\lambda}{2r}(r^2-\|x-x_0\|^2),\;\|x-x_0\|\leq r,\\
                   0,\;\;\;\;\;\;\;\;\;\;\;\;\;\;\;\;\;\;{\rm otherwise}.\;\;
                  \end{array}
\right.
\end{equation}
Then
\begin{equation}\sum\limits_{i=1}^{n}(\frac{\partial u}{\partial x_i})^2-1=\left\{\begin{array}{ccc}
                   4\lambda^2\|x-x_0\|^2-1,\;\|t-t_0\|\leq r,\\
                   -1,\;\;\;\;\;\;\;\;\;\;\;\;\;\;\;\;\;\;\;\;\;\;\;{\rm otherwise}.\;\;\;\;\;
                  \end{array}
\right.
\end{equation}
Note that $v\in{\rm int}L^+_\infty(\Omega)$ if and only if there is $\alpha>0$ such that $\mid v\mid\geq\alpha$ a.e. $\Omega$.
 Then $\sum\limits_{i=1}^{n}(\frac{\partial u}{\partial x_i})^2-1\in -{\rm int}L^+_\infty(\Omega)$. If program (\ref{elastic plastic torison problem}) is solvable with a solution $\bar u$, then by Theorem \ref{10.5}, we know the following duality program is solvable with same extreme value of the program (\ref{elastic plastic torison problem}).

\begin{equation}
\begin{split}
&\max\limits_{y^*}\inf\limits_{u\in X_\pi} f(\bar u,u)+\left<y^*,\;(\sum\limits_{i=1}^{n} \frac{\partial u}{\partial x_i})^2-1\right>\\
&s.t. \quad\quad y^*\in L_{\infty}^{*+}(\Omega),
\end{split}
\end{equation}
where $X_\pi=\bigcup_{n=1}^\infty nK$.
\qed
\end{example}

\section{Convex vector optimization}
In this section, we will apply the theory of generating spaces to convex vector programs with box constraints in (infinite dimensional) Banach spaces.\par
Let us come back to the program (\ref{vector model}) with the box constraints $\Omega\subset \{x\in X:\;x_a\leq x\leq x_b\}$.
\begin{equation}\label{11.1}
\begin{split}
&\min\quad f(x)\\
&s.t.\quad\; x\in \Omega\subset \{x\in X:\;x_a\leq x\leq x_b\},\\
\end{split}
\end{equation}
where $X,\;Y$ are Banach spaces ordered by their ordering cones $X^+,\;Y^+$ respectively with $(X^+)_N\neq \emptyset\neq (Y^+)_N$. $\Omega$ is a closed convex subset with some $x_a<x_b\in X$. $f:X\rightarrow Y$ is $Y^+$-Lipschitz on $\Omega$ with respect to some $e\in (Y^+)_N$ (Definition \ref{9.1}), $\| e \|=1$, and is convex with respect to $Y^+$. \par
Usually, the first step to deal with the program (\ref{11.1}) is to consider a scalarized program. There are many scalarization methods (See,  for instance,  \cite{18Ansari,08Eichfelder,20Tammer}).  In this section, we will use
the following Gerstewitz scalarization function $\phi_{e,C}:Y\rightarrow \mathbb{R}\cup\{+\infty\}$ respect to  $C=Y^+$ and $e$  is defined for $y\in Y$ by
\begin{equation}\label{11.2}\phi_{e,C}(y)=\inf\{t\in \mathbb{R}:te\in y+C\}.\end{equation}
Note that $te\in y+C$ is equivalent to $0\in y-te+C$. Then
\begin{equation}\label{11.3}\phi_{e,C}(y)=\inf\{t\in \mathbb{R}:y-te\in -C\}.\end{equation}

Keep these notions in mind. We have following properties  \cite{20Tammer}.\par
\begin{lemma}\label{properties of scalarization function} Let $Y$ be a Banach space and $C=Y^+$ be its ordering cone.

(1) $\phi_{e,C}$ is lower semicontinuous if and only if $C$ is closed in $Y$.\par
(2) $\phi_{e,C}$ is a monotone nondecreasing sublinear functional with respect to $Y^+$.\par
(3) For all $y\in Y$, $r\in \mathbb{R}$, $\phi_{e,C}(y)\leq r$ $\Longleftrightarrow$ $y\in re-C$, in particular, $\phi_{e,C}(y)\leq 0$ if and only if $y\in -C$.\par
(4) $\phi_{e,C}$ is continuous on $Y$ if and only if $e\in {\rm int}(C)$.\par
(5) The effective domain of $\phi_{e,C}$ is $\mathbb{R}(e-C)\equiv\bigcup_{r\in\mathbb R}r(e-C)$.\par
\end{lemma}

We consider following scalarization program. Given $\bar{x}\in \Omega$, \par
\begin{equation}\label{scalarization program}
\begin{split}
&\min\quad \phi_{e,Y^+}(f(x)-f(\bar{x}))\\
&s.t.\quad\; x\in \Omega.
\end{split}
\end{equation}
It is obvious that if $\bar{x}$ is a local minimizer of program (\ref{11.1}) if and only if $\phi_{e,Y^+}(f(x)-f(\bar{x}))\geq 0$. In this case, $\bar{x}$ is a minimizer of program (\ref{scalarization program}). \par
The most fundamental and important optimality condition induced by Fermat rule for this program is
\[\theta\in \partial \phi_{e,Y^+}(f(.)-f(\bar{x})+\delta(.,\Omega))(\bar{x}),\]
where $\delta(.,\Omega)$ is the indicator function of $\Omega$, that is,  $\delta(x,\Omega)=0,\;{\rm if}\;x\in\Omega;\;=\infty,\;{\rm otherwise}.$
 Since the Moreau-Rockafellar theorem requires one of $\phi_{e,Y^+}$ and $\delta(\cdot,\Omega)$ is continuous, the following equation
\[\partial (\phi_{e,Y^+}(f(\cdot)-f(\bar{x}))+\delta(\cdot,\Omega))(\bar{x})=
\partial \phi_{e,Y^+}(f(\cdot)-f(\bar{x}))(\bar{x})+\partial\delta(\cdot,\Omega)(\bar{x}),\]
 does not hold if int$(Y^+)=\emptyset$.  This is an essential difficulty and does not depend on the choice of  scalarization methods. In order to circumvent this difficulty, the next best thing is  to replace $Y^+$ by a larger set $A\supset Y^+$  with int$(A)\neq \emptyset$. See, for instance, \cite{18Ansari,21Amahroq,08Eichfelder,10Kasimbeyli,16Pirro,20Tammer,10Truong} and  references therein. Nevertheless, by this procedure, one would obtain ``approximate minimal point", instead of minimum point of (\ref{11.1}). In the following, we will show that this difficulty can be overcome by the generating space theory whenever $(Y^+)_N\neq\emptyset$.

 Assume that the Banach spaces $X,Y$, the domain $\Omega$, the function $f$ and the vector $e\in(Y^+)_N$ as the same as in the program (\ref{11.1}).
For a fixed $\bar x\in \Omega$, denote $\bar{f}(x)=f(x)-f(\bar{x})$. Let $Y_e$ be $(Y^+,e)$-generating space with its positive cone $Y_e^+$ and the dual positive cone $Y_e^{*+}$. Note that $e\in{\rm int}Y_e^+$.
\begin{theorem}\label{new fermat rule} With notions and symbols as above,
suppose that $\bar x$ is a minimum of the program (\ref{scalarization program}). Then
\begin{equation}\theta\in \partial \phi_{e,Y_{e}^+}\circ \bar{f}(\bar x)+N(\bar x,\Omega),\end{equation}
where $N(\bar x,\Omega)$ is from Definition \ref{definition of subdifferential}.
\end{theorem}
\begin{proof}
Since $f$ is $Y^+$-Lipschitz on $\Omega$ with respect to $e\in (Y^+)_N$, we see that \[\{f(x)-f(\bar{x}):x\in \Omega\} \subset Y_e,\] and that  $\bar f$ is   $Y^+_e$-Lipschitz on $\Omega$. Note that  $e\in{\rm int}(Y^+_e)$. Then by Lemma \ref{properties of scalarization function}, $\phi_{e,Y_e^+}:(Y_e,\| . \|_e)\rightarrow {\mathbb{R}}$ is a continuous convex function. By the Fermat rule, the Moreau-Rockafellar Theorem (Theorem \ref{Moreau-Rockafellar theorem}) and the equation (\ref{delta and N}), we obtain
\begin{equation}\nonumber
\theta\in\partial (\phi_{e,Y_e^+}\circ\bar f+\delta(.,\Omega))(\bar{x})\subset
\partial \phi_{e,Y_{e}^+}\circ\bar f(\bar{x})+N(\bar x,\Omega), \qed
\end{equation}
\end{proof}


\begin{lemma}[\cite{10Tammer}]\label{subdifferential of phi} Assume that $Y$ is a Banach space, $e\in Y\setminus\{\theta\}$, and $C\subset Y$ is a closed convex  cone  not containing the real line $\mathbb{R}e$ with $C+[0,\infty)e\subset C$. Let $C^*$ be the dual positive cone, that is, $C^*\equiv\{y^*\in Y^*:\;\left<y^*, y\right>\geq 0,\;\forall y\in C\}$, and
$\phi_{e,C}:Y\rightarrow \mathbb{R}\cup\{+\infty\}$ be the corresponding Gerstewitz scalarization function.   Then
\begin{equation}\partial \phi_{e,C}(y)\neq\emptyset,\:\;{\rm for\;all\;} y\in{\rm dom}\phi_{e,C}\end{equation} and
\begin{equation}\partial \phi_{e,C}(y)=\{y^*\in C^*:\left<y^*,\;e\right>=1,\; \left<y^*,\;y\right>=\phi_{e,C}(y)\}.\end{equation}
In particular,  \begin{equation}\partial \phi_{e,C}(\theta)=\{y^*\in C^{*}:\left<y^*,\;e\right>=1\}.\end{equation}
\end{lemma}

The next result is a representation theorem of $\partial \phi_{e,Y_e^+}\circ \bar f(\bar{x})$. \par

\begin{theorem}\label{subdifferential of composition} With the notions and conditions as previously mentioned, suppose that $\bar x$ is a minimum of the program (\ref{scalarization program}). Then
\begin{equation}\label{11.1'}
\begin{split}
\partial \phi_{e,Y_e^+}\circ\bar{f}(\bar{x})&\subset\bigcup\{\partial\left<y^*,\bar{f}\right>(\bar{x}):y^*\in Y^{*+}_e,\;\left<y^*,e\right>=1\}\\&=\overline{\rm co}^*\big(\bigcup\{\partial\left<y^*,\;\bar{f}\right>(\bar{x}):y^*\in K_e\}\big),
\end{split}
\end{equation}
where $K_e={\rm extr}\{y^*\in Y^{*+}_e:\;\left<y^*,e\right>=1\}$, the set of all extreme points of the $w^*$-closed convex set $Y^{*+}_e\bigcap\{y^*\in Y^{*}_e: \langle y^*.e\rangle=1\}$, and $\overline{\rm co}^*(A)$ denotes the $w^*$-closed convex hull of a set $A\subset Y_e^*$.
\end{theorem}

\begin{proof}
Since $f$ is $Y^+$-Lipschitz with respect to $e$, $\{f(x)-f(\bar{x}):x\in \Omega\} \subset Y_e$. Since $f$ is convex with respect to $Y^+$,   $\bar f$ is convex with respect to $Y_e^+$.
Note that $e\in{\rm int}Y_e^+$. Then by Lemma \ref{properties of scalarization function}, $\phi_{e,Y_e^+}$ is a continuous sublinear and monotone nondecreasing functional with respect to the ordering cone $Y^+_e$ on $Y_e$.
Let $\bar{y}=\bar{f}(\bar{x})$. Since $\bar x$ is a minimum of the program (\ref{scalarization program}),  nondecreasing monotonicity of $\phi_{e,Y_e^+}$ entails
\begin{equation}\label{11.4}
\begin{split}
\phi_{e,Y_e^+}\circ \bar{f}(\bar{x})&=\inf_{y\in Y_e}\{\phi_{e,Y_e^+}(y)+\delta\big((\bar{x},y);{\rm epi}\bar{f}\big)\}\\
&=\phi_{e,Y_e^+}(\bar{y})+\delta\big((\bar{x},\bar{y});{\rm epi}\bar{f}\big).
\end{split}
\end{equation}
Since $\phi_{e,Y_e^+}$ is a continuous sublinear functional and since $\delta\big((\bar x,\cdot);{\rm epi}\bar{f}\big)$ is a lower semicontinuous convex function on $Y_e$, by the Moreau-Rockafellar Theorem,
\begin{equation}
\begin{split}
\partial\phi_{e,Y_e^+}\circ \bar{f}(\bar{x})&=\partial\phi_{e,Y_e^+}(\bar{y})+\partial\delta\big((\bar{x},\bar{y});{\rm epi}\bar{f}\big)\\
&=\{x^*\in X^*: y^*\in \partial \phi_{e,Y_e^+}(\bar{y}),(x^*,-y^*)\in N((\bar{x},\bar{y}),{\rm epi}\bar{f})\}.
\end{split}
\end{equation}

Note that for every $(x^*,-y^*)\in N((\bar{x},\bar{y}),{\rm epi}\bar{f})$, and for every pair $(x,y)\in X\times Y_e$, we have
\[\delta((x,y),{\rm epi}\bar{f})-\delta((\bar{x},\bar{y}),{\rm epi}\bar{f})\geq \left<x^*, x-\bar{x}\right>-\left<y^*,y-\bar{y}\right>.\]
Then for all $(x,y)\in {\rm graph}\bar f$,
\begin{equation}\nonumber
\begin{split}
0&\geq \left<x^*, x-\bar{x}\right>-\left<y^*,y-\bar{y}\right>\\
&=\left<x^*, x-\bar{x}\right>-\left<y^*,\bar{f}(x)-\bar{f}(\bar{x})\right>.
\end{split}
\end{equation}
Therefore,\begin{equation}x^*\in\partial\left<y^*,\bar{f}\right>(\bar{x}).\end{equation}
 This and Lemma \ref{subdifferential of phi} entail
\begin{equation}\label{11.5}
\begin{split}
\partial (\phi_{e,Y_e^+}\circ\bar{f}(\bar{x}))
\subset\bigcup\{\partial\left<y^*,\bar{f}\right>(\bar{x}):y^*\in \partial\phi_{e,Y_e^+}(\theta)\}\\
=\bigcup\{\partial\left<y^*,\bar{f}\right>(\bar{x}):y^*\in Y^{*+}_e,\;\left<y^*,e\right>=1\}.
\end{split}
\end{equation}

Since $\partial \phi_{e,Y_e^+}(\theta)=\{y^*\in Y_e^{*+}:\;\left<y^*,\;e\right>=1\}$ is a $w^*$-compact convex set in $Y_e^{*+}$. By the Krein-Milman theorem,
\[\partial \phi_{e,Y_e^+}(\theta)=\overline{{\rm co}}^*\big({\rm extr}(\partial\phi_{e,Y_e^+}(\theta))\big)=\overline{\rm co}^*(K_e).\]
 Therefore,  for all $m\in\mathbb N$, $y_i^*\in(K_e)$, $\lambda_i\geq0,\;i=1,...m$ with $\sum\limits_{i=1}^{m}\lambda_i=1$,   and $y^*=\sum\limits_{i=1}^{m}\lambda_iy^*_i$, we obtain
\begin{equation}\nonumber
\begin{split}
\partial\left<y^*,\;\bar{f}\right>(\bar{x})=\partial\left<\sum\limits_{i=1}^{m}\lambda_iy^*_i,\;\bar{f}\right>(\bar{x})=
\sum\limits_{i=1}^{n}\lambda_i\partial \left<y^*_i,\;\bar{f}\right>(\bar{x}).
\end{split}
\end{equation}
Consequently,
\begin{equation}\label{11.6}
\begin{split}
\bigcup\{\partial\left<y^*,\bar{f}\right>(\bar{x}):\;y^*\in\partial\phi_{e,Y^+_e}(\theta)\}
=\overline{\rm co}^{*}(\bigcup\{\partial\left<y^*,\;\bar{f}\right>(\bar{x}):y^*\in K_e\}).
\end{split}
\end{equation}
(\ref{11.5}) and (\ref{11.6}) together imply  (\ref{11.1'}). \qed
\end{proof}

As an application of Theorem \ref{subdifferential of composition}, we will give a necessary condition for a vector variational inequality in infinite dimensional Banach spaces.

Vector variational inequalities have been widely studied since  they were introduced by Giannessi \cite{80Giannessi} in 1980. See, for example,  \cite{12Ansari,18Ansari,00Giannessi,05Chen}. To due with vector variational inequalities  in infinite dimensional spaces, one would encounter significant difficulties of the non-solidness of ordering cones in question. Especially,  scalarization methods in vector variational inequalities are surprisingly restricted.  (See, for example, Ansari  \cite{18Ansari}.)

For a Banach space $Z$, we denote by $B(Z)$ the space of all bounded linear operators from $Z$ to itself.
\begin{example}\label{variational inequality}
Let $(K,\sum,\mu)$ be a probability space. Assume that $x_a,\;x_b,\;x_b-x_a\in (L_2^+(\mu))_N$, and that $\Omega\equiv\{x\in L_2(\mu):\;x_a\leq x\leq x_b\}$ is a closed convex set. Suppose $T:L_2(\mu)\rightarrow B(L_2(\mu))$ is a bounded linear operator, which satisfies that for each fixed $x\in L_2(\mu)$, $Tx$ is  $L_2(\mu)$-Lipschitz on $\Omega$ with respect to some $e\in (L_2(\mu))^+_N$ with $\| e\|=1$. We use $y\nleqslant_{L_2^+\setminus \{0\}}0$ to denote that $y\notin -L_2^+\setminus \{0\}$.
Then the variational inequality problem is to find a $\bar x\in \Omega$ such that
\begin{equation}\label{vector variational inequality}
\begin{split}
\left<T\bar x, x-\bar x \right>\nleqslant_{L_2^+\setminus \{0\}}0,\;{\rm for\;all}\; x\in\Omega,
\end{split}
\end{equation}
where $\left<T\bar x, x-\bar x \right>=(T\bar x)(x-\bar x).$

 Suppose $\bar x\in \Omega$ is a solution to the program (\ref{vector variational inequality}). Then $\bar x$ is a minimum of following program.
\begin{equation}\nonumber
\begin{split}
&\min\;\;\left<T \bar x, x-\bar x \right>\\
&s.t.\;x\in \Omega.
\end{split}
\end{equation}
Let $X\equiv L_2(\mu)=Y$.  Since $x_a,\;x_b,\;x_b-x_a\in (L_2(\mu)^+)_N$,  $\pi\equiv\frac{x_b}{\| x_b\|}\in(L_2(\mu)^+)_N$. Let $X_\pi$ be the $(X^+,\pi)$-generating space. Then $\Omega\subset X_\pi$, and
(by Example \ref{4.1}) $X_\pi=L_\infty(\mu)$. Since $e\in (L_2(\mu))^+_N\equiv Y^+$, the $(Y^+,e)$-generating space $Y_e$ is also $L_\infty(\mu)$.
Since $T$ is $L_2(\mu)$-Lipschitz on $\Omega$ with respect to $e\in (L_2(\mu))^+_N$, we obtain that $T(\Omega)\subset Y_e=L_\infty(\mu).$
 Therefore,  $X_\pi^*=(L_\infty(\mu))^*=Y_e^*$. Consequently,  $\bar x$ is a solution of the following scalarization program.
\begin{equation}\nonumber
\begin{split}
&\min\;\;\phi_{e,L_\infty(\mu)^+}(\left<T \bar x, x-\bar x \right>)\\
&s.t.\;x\in \Omega\subset L_\infty(\mu).
\end{split}
\end{equation}
By Theorems \ref{new fermat rule} and  \ref{subdifferential of composition},
\begin{equation}\nonumber
\begin{split}
\theta&\in\partial \phi_{e,L_\infty(\mu)^+}(\left<T\bar x, \cdot\right>)(\bar x)+N(\bar x,\Omega)\\
&=\bigcup\big\{y^*\circ T(\bar x):\;\left<y^*,\;e\right>=1,\;y^*\in L_\infty(\mu)^{*+}\big\}+N(\bar x,\Omega).
\end{split}
\end{equation}
This is a necessary optimality condition of the program (\ref{vector variational inequality}).\qed
\end{example}

\begin{acknowledgements}
The work was supported by National Natural Science Foundation of China, grant No. 11731010.
\end{acknowledgements}

%



\end{document}